\documentclass[11pt,a4paper,reqno]{amsart}            

\usepackage[english]{babel}
\usepackage{amssymb,amsmath,amsthm}
\usepackage{mathrsfs}
\usepackage{enumitem}
\usepackage[immediate,debrief]{silence}
\usepackage[pagebackref]{hyperref}
\usepackage{bbm}
\usepackage{csquotes}

\usepackage{color}
\usepackage{comment}
\usepackage{scrtime}

\hypersetup{
 colorlinks,
 linkcolor={blue!90!black},
 citecolor={red!80!black},
 urlcolor={blue!50!black}
}
\usepackage{tikz}
\usepackage{tikz-cd}
\usepackage{graphicx}
\usetikzlibrary{shapes.geometric,arrows,fit,matrix,positioning,trees,snakes}
\tikzset
{
    treenode/.style = {circle, draw=black, align=center, minimum size=1cm},
    subtree/.style  = {isosceles triangle, draw=black, align=center, minimum height=0.5cm, minimum width=1cm, shape border rotate=90, anchor=north}
}

\renewcommand\mathcal{\mathscr}
\theoremstyle{plain}
\newtheorem{theorem}{Theorem}[section]
\newtheorem*{theorem*}{Theorem}
\newtheorem{lemma}[theorem]{Lemma}
\newtheorem*{lemma*}{Lemma}

\newtheorem{proposition}[theorem]{Proposition}

\newtheorem*{conjecture*}{Conjecture}

\newtheorem*{notation*}{Notation}

\theoremstyle{remark}
\newtheorem{remark}[theorem]{Remark}
\newtheorem*{remark*}{Remark}

\theoremstyle{definition}
\newtheorem{definition}[theorem]{Definition}
\newtheorem*{definition*}{Definition}

\theoremstyle{example}

\newtheorem*{example*}{Example}

\numberwithin{equation}{section}

\newcount\quantno
\everydisplay{\quantno=0}\everycr{\quantno=0}
\newcommand\quant{\advance\quantno by1
                      \ifnum\quantno=1\qquad\else\quad\fi\forall }
\newcommand\itemno[1]{(\romannumeral #1)}

\renewcommand\Re{\operatorname{\mathrm{Re}}}
\renewcommand\Im{\operatorname{\mathrm{Im}}}

\newcommand\rest[1]{\kern-.1em
          \lower.5ex\hbox{$\scriptstyle #1$}\kern.05em}

\renewcommand\mod[1]{\vert{#1}\vert}
\newcommand\bigmod[1]{\bigl\vert{#1}\bigr|}
\newcommand\Bigmod[1]{\Bigl\vert{#1}\Bigr|}

\newcommand\norm[2]{{\Vert{#1}\Vert_{#2}}}
\newcommand\normto[3]{{\Vert{#1}\Vert_{#2}^{#3}}}
\newcommand\bignorm[2]{\left.{\bigl\Vert{#1}\bigr\Vert_{#2}}\right.}
\newcommand\bignormto[3]{\left.{\bigl\Vert{#1}\bigr\Vert_{#2}^{#3}}\right.}
\newcommand\Bignorm[2]{\left.{\Bigl\Vert{#1}\Bigr\Vert_{#2}}\right.}

\newcommand\prodo[2]{\left\langle#1,#2\right\rangle}

\newcommand\wrt{\,\text{\rm d}}

\newcommand\bD{\mathbf{D}}

\newcommand\BH{\mathbb{H}}

\newcommand\BR{\mathbb{R}}

\newcommand\BZ{\mathbb{Z}}

\newcommand\cA{\mathcal{A}}   
\newcommand\fra{\mathfrak{a}}  
\newcommand\cB{\mathcal{B}}  \newcommand\fB{\mathfrak{B}}  \newcommand\fb{\mathfrak{b}}  
\newcommand\frb{\mathfrak{b}} 
\newcommand\cC{\mathcal{C}}  \newcommand\fC{\mathfrak{C}}

  \newcommand\fF{\mathfrak{F}} 
 
\newcommand\cG{\mathcal{G}}  \newcommand\fG{\mathfrak{G}} 
 
\newcommand\cH{\mathcal{H}}

\newcommand\cM{\mathcal{M}}     
\newcommand\cN{\mathcal{N}}   \newcommand\frn{\mathfrak{n}}

  \newcommand\fQ{\mathfrak{Q}}   
\newcommand\cR{\mathcal{R}}    
   \newcommand\frs{\mathfrak{s}} 
\newcommand\cT{\mathcal{T}}  \newcommand\fT{\mathfrak{T}}  
\newcommand\cU{\mathcal{U}}    
   \newcommand\frv{\mathfrak{v}}

   \newcommand\frz{\mathfrak{z}}

\newcommand\al{\alpha}
\newcommand\be{\beta}
\newcommand\ga{\gamma}    
\newcommand\de{\delta}

\newcommand\la{\lambda}   
\newcommand\om{\omega}    \newcommand\Om{\Omega}  
    
\newcommand\te{\theta}

\newcommand\funnyk{k\hbox to 0pt{\hss\phantom{g}}}

\newcommand\lu[1]{L^1(#1)}
\newcommand\lp[1]{L^p(#1)}

\newcommand\ly[1]{L^\infty(#1)}

\newcommand\whH{\widehat{\phantom{G}}\hbox to 0pt{\hss $H$}}

\newcommand\emspace{\hbox to 6pt{\hss}}
\newcommand\ds{\displaystyle}

\newcommand\rmi{\hbox{\rm (i)}}
\newcommand\rmii{\hbox{\rm (ii)}}
\newcommand\rmiii{\hbox{\rm (iii)}}

\newcommand\ioty{\int_0^{\infty}}

\newcommand\One{{\mathbbm{1}}}

\newcommand\e{\mathrm{e}}

\newcommand\arcsinh{\mathrm{arsinh}}

\newcommand\ucT{\underline{\cT}}
\newcommand\ocT{\overline{\cT}}
\newcommand\matrice[4]{\left[\begin{matrix}#1 &#2 \\#3 &#4\end{matrix}\right]}

\DeclareSymbolFont{EUEX}{U}{euex}{m}{n}

\DeclareSymbolFont{euexlargesymbols}{U}{euex}{m}{n}
\DeclareMathSymbol{\intop}{\mathop}{euexlargesymbols}{"52}
     \def\int{\intop\nolimits}

\DeclareSymbolFont{euexsymbols}     {U}{euex}{m}{n}
\DeclareMathSymbol{\smallint}{\mathop}{euexsymbols}{"52}

\begin{document}

\title[]{Uncentred maximal operators with respect to half balls on Damek--Ricci spaces}



\thanks{The first and fourth authors are members of the INdAM group GNAMPA and
are partially supported by the grant \enquote{INdAM-GNAMPA Project}, CUP
E53C25002010001, \enquote{Transferring Harmonic Analysis between Discrete
Structures and Manifolds}. The first author was partially supported by
the research grant \enquote{Yields of the ubiquity and the geometry of inner
functions (YoungInFun)}, PID2024-160326NA-I00. The third author is supported by the Deutsche Forschungsgemeinschaft
(DFG, German Research Foundation)–SFB-Gesch{\"a}ftszeichen –Projektnummer SFB-TRR
358/1 2023 –491392403.
Part of this paper was done while the second author was visiting Universit{\"a}t Paderborn and while the third author was visiting Università di Milano-Bicocca. Both travelers would like to thank the host institutions for the hospitality.}

\author[]{Nikolaos Chalmoukis, Stefano Meda,
Effie Papageorgiou \\ and Federico Santagati}

\address[Nikolaos Chalmoukis]{Dipartimento di Matematica e Applicazioni \\ Universit\`a di Milano-Bicocca\\
via R.~Cozzi 53\\ I-20125 Milano\\ Italy}
\email{nikolaos.chalmoukis@unimib.it}

\address[Stefano Meda]{Dipartimento di Matematica e Applicazioni \\ Universit\`a di Milano-Bicocca\\
via R.~Cozzi 53\\ I-20125 Milano\\ Italy}
\email{stefano.meda@unimib.it}

\address[Effie Papageorgiou]{Institut f\"ur Mathematik \\ Universit\"at Paderborn\\
D-33098 Paderborn \\ Germany}
\email{papageoeffie@gmail.com}

\address[Federico Santagati]{Dipartimento di Scienze Matematiche ``Giuseppe Luigi Lagrange" \\ Politecnico di Torino\\
C.so Duca degli Abruzzi 24 \\ 10129 Torino\\ Italy}
\email{federico.santagati@polito.it}

\begin{abstract}
	In this paper we study a variant of the uncentred Hardy--Littlewood maximal operator on Damek--Ricci spaces in which
	balls are replaced by suitable half balls. Perhaps surprisingly, such modified maximal operator has better 
	boundedness properties than the classical one.  {In particular, it is bounded on $L^p$ 
	for every $p$ in $(1,\infty]$ (whereas the analogue operator on balls is bounded on $L^p$ only for $p>2$),
	and satisfies a limiting distributional inequality if $f$ is in $L\log (2+L)$}.
	{This endpoint estimate is optimal in the sense that it does not hold if $L\log ({2}+L)$ 
	is replaced by a larger (in a suitable sense) Orlicz space}.
\end{abstract}


\maketitle

\section{Introduction} \label{s: Introduction}
Given a metric measure space $(X,d,\mu)$, a family $\fF$ of subsets of $X$ of positive measure, a locally integrable 
function $f$ on $X$ and a point $x$ in $X$, define the \textit{uncentred Hardy--Littlewood maximal operator} associated to $\fF$ by
\begin{equation} \label{f: HLMO variants}
\cN^\fF\! f(x)
	:= \sup_{F\ni x\atop F\in\fF} \, \frac{1}{\mod{F}} \, \int_F \mod{f} \wrt \mu.
\end{equation}
The boundedness properties of $\cN^\fF\!$ between $L^p$, Lorentz and Orlicz spaces and for a wide variety of metric measure
spaces $X$ and families $\fF$ have been investigated in several papers.

In this paper we suppose that $S$ is a Damek--Ricci space endowed with a left-invariant metric Riemannian $d$ 
and Riemannian measure $\mu$ and analyse the boundedness properties of $\cN^\fF$ for various families $\fF$ of subsets of $S$. 
We recall that the metric measure space $(S,d,\mu)$ has exponential volume growth, and it is locally, but not globally, doubling.
For any measurable subset $E$ of $S$ we shall often write~$\mod{E}$ instead of $\mu(E)$. 

Recall that the Damek--Ricci spaces include all the symmetric spaces of the noncompact type and real rank one, and, in particular,
the real hyperbolic spaces.  The paradigmatic example of the latter is the hyperbolic upper half plane, which will play
an important role in this paper.  For more on the structure of Damek--Ricci spaces and their relation with noncompact
symmetric spaces, see \cite{DR, CDKR1, CDKR2, R}.  We emphasise the fact that most Damek--Ricci spaces are not symmetric.  

Denote by $\fB$ the family of all open balls in $S$, and consider, for each locally integrable function $f$ on $S$, its 
classical \textit{uncentred} Hardy--Littlewood (HL) maximal function $\cN^\fB\! f$. 
A well known result by A.~Ionescu \cite{I1} states that if~$S$ is a symmetric space of the 
noncompact type and real rank one, then $\cN^\fB$ is of restricted weak type $(2,2)$, and, interpolating with the trivial bound 
on $\ly{S}$, also bounded on $\lp{S}$ for $p>2$.  He also proved that on any higher rank symmetric space $\cN^\fB$ is bounded on 
$\lp{S}$ if and only if $p>2$ \cite{I2}. For an extension of the rank one result \cite{I1} to homogeneous trees, see \cite{V}.  
Extensions of Ionescu's results to Riemannian manifolds with exponential volume growth have been proved by various authors. 
See, in particular, \cite{L2, L3}, where certain cuspidal manifolds are considered, and \cite{MPSV}.  

It is worth mentioning that the \textit{centred} Hardy--Littlewood maximal operator $\cM$, defined much as $\cN^\fB$, 
but where the supremum is taken over all balls centred at $x$, has also been the object of several studies. After the landmark 
results of Str\"omberg \cite{Str}, who proved the weak-type $(1,1)$ estimate for $\cM$ on all symmetric spaces 
of the noncompact type, quite a few authors proved interesting results concerning $\cM$ on various metric measure spaces
with exponential volume growth. These include \cite[Corollary 3.22]{ADY}, where Str\"omberg's result is extended 
to Damek–Ricci spaces, \cite{Str}, where an interesting Cartan--Hadamard surface is considered,
\cite{L1,L2,L3} where certain cuspidal manifolds are studied, \cite{CMS} where the quite general
case of Gromov hyperbolic spaces satisfying certain exponential volume growth estimates is treated, and \cite{MPSV}, where
the invariance of the $L^p$ mapping properties of $\cM$ with respect to rough isometries is proved.

For results in the case of trees, which motivated further analysis in the continuous case, see \cite{LMSV, LS, MS, CMS1, NT}.  

The purpose of this paper is to consider a variant $\cN^\frb$ of $\cN^\fB$ in which balls are replaced by certain half balls, 
and to show that, perhaps surprisingly, $\cN^\frb$ has better boundedness properties than $\cN^\fB$.  This work is inspired by 
\cite{MS}, where the case of (possibly nonhomogeneous) locally finite trees with exponential volume growth is considered.  
However, their extension to Damek--Ricci spaces requires new ideas and is technically much more involved.  It would be 
interesting to extend our result to Cartan--Hadamard manifolds, but this will require different methods.  

{We observe that this phenomenon does not occur on {$\BR^n$, endowed with the Euclidean distance and
any doubling measure}, where $\cN^{\fB}$ and
$\cN^{\fb}$ are pointwise equivalent, hence share the same $L^p$ boundedness properties.}

It is perhaps convenient to first describe $\cN^\frb$ in the simplest possible case, 
viz. the half plane $\big\{(x,y) \in \BR^2: y>0 \big\}$, endowed with the {
Riemannian metric 
\begin{equation} \label{f: hyp metric}
	\wrt s^2
	:= \frac{\wrt x^2 + \wrt y^2}{y^2}
\end{equation}
and measure 
$$
	\wrt \mu(x,y) = \frac{\wrt x\wrt y}{y^2}.
$$
}
Horizontal lines in the upper half plane are the horocycles with respect to the point at infinity in the direction of the $y$-axis.  
Given a point $z$ in $\BH$ we consider the unique geodesic tangent to the horocycle through $z$.  This geodesic splits
the ball $B_R(z)$ (centred at $z$ and with radius $R$) into two parts, the lower of which we call $b_R(z)$. The family
of such half balls will be denoted by $\frb$.  Since $\BH$ is symmetric, the measure of $b_R(z)$ is just one half the measure of 
$B_R(z)$.  

These sets and the family $\frb$ have natural generalisations on any Damek--Ricci space:
see Section~\ref{s: Damek--Ricci spaces} for the details.  The 
uncentred maximal operator on $S$ associated to the family $\frb$ is then defined as in \eqref{f: HLMO variants}, with
$\frb$ in place of~$\fF$.  
Since for every $x$ in $S$ and every positive radius $R$ the measure of $b_R(x)$ is one half of the measure of $B_R(x)$, 
a straightforward argument yields the pointwise estimate
$$
\cN^\frb\! f \leq 2 \, \cN^\fB\! f,
$$
which, by (a straightforward generalisation of) Ionescu's result, implies that~$\cN^\frb$ is bounded on 
$\lp{\BH}$ for all $p>2$ and it is of restricted weak type $(2,2)$.

The main result of the paper, which is novel also for the hyperbolic upper half plane, is the following, {where 
$\Phi(t) := t \, \log(2+t)$.}

\begin{theorem} \label{t: main}
	Suppose that $S$ is a Damek--Ricci space.  Then the following hold:
	\begin{enumerate}
		\item[\itemno1]
			there {exists a positive constant} $C$ such that 
			$$
				\bigmod{\{x\in S: \cN^\frb\! f(x)>\al \}}
				\leq C \, \int_{S} \, {\Phi\big(\mod{f}/\al}\big) \wrt \mu
			$$
			for every $\al >0$ and every {measurable function} $f$ {for which the right hand 
			side is finite};
		\item[\itemno2]
			$\cN^\frb$ is bounded on $\lp{S}$ for every $p$ in $(1,\infty]$;  
		\item[\itemno3]
			{ statement (i) is optimal in the sense that we cannot replace $\Phi$ 
			by another Young function $\Psi$ on $(0,\infty)$ such that 
			$\ds 
			\lim_{t\to+\infty} \, \frac{\Psi(t)}{\Phi(t)} = 0.
			$}
			{In particular,} $\cN^\frb$ is not of weak-type $(1,1)$.   
	\end{enumerate}
\end{theorem}

\noindent

Notice that only half balls with large radii (say greater than one) really matter.  Indeed, denote by 
$\frb_0$ and~$\frb_\infty$ the collection of all half balls~$b_R(x)$ for which $R\leq 1$ and $R>1$, respectively, 
and define, much as for~$\frb$, the operators $\cN^{\frb_0}$ and $\cN^{\frb_\infty}$.  These operators are the
\textit{local} and the \textit{global} uncentred maximal operators associated to the family $\frb$.
A straightforward argument based on a partition of unity on $S$ associated to balls of radius $1$ with 
the finite overlapping property and the local doubling property of the Riemannian measure shows that the operator $\cN^{\frb_0}$ 
is of weak type $(1,1)$ and bounded on $\lp{S}$ for every $p$ in $(1,\infty]$. Thus, we focus on the operator~$\cN^{\frb_\infty}$,
and it suffices to prove Theorem~\ref{t: main} with $\frb_\infty$ in place of $\frb$. 


The main idea to prove Theorem~\ref{t: main} is to rely on the geometric approach to the $L^p$ boundedness of the uncentred maximal 
operator developed by A.~C\'ordoba and R.~Fefferman \cite{CF}.  Loosely speaking, given a 
family $\fF$ of subsets of a metric measure space $X$, the $\lp{X}$ boundedness of $\cN^{\fF}$ for $p$ in $(1,\infty)$ 
is equivalent to a variant of a classical covering lemma in which mutual disjointness of subsets in any
appropriate subcollection $\fG'$ of a given collection $\fG$ is replaced by appropriate estimates of 
\begin{equation} \label{f: CF inequality}
	\Bignorm{\sum_{b\in\fG'}\, \One_b\,}{p'}
	\leq A_p \bignorm{\One_G}{p},
\end{equation}
where $G$ denotes the union of all $b$ in $\fG$ and $p'$ the index conjugate to $p$.
This estimate says that the balls in~$\fG'$ have ``finite overlapping in the~$L^{p'}$ norm''.  By Ionescu's countererexample
and the C\'ordoba--Fefferman result, the analogue of \eqref{f: CF inequality} for balls fails for~$p'$ in $(2,\infty)$.  
So, retrospectively, our result says that on Damek--Ricci spaces half balls intersect ``much less'' than balls.  
This phenomenon does not occur on Euclidean spaces.  

To prove \eqref{f: CF inequality}, we first show that $\cN^{\frb_\infty}$ can be estimated from above by an uncentred 
maximal operator with respect to a different family of sets, and then show that this new family of sets satisfies
\eqref{f: CF inequality}. 

Our result has a geometric flavour that can be appreciated by a possibly wider audience in the special case of the 
upper half plane.  Indeed, in this case the sets involved in the proof of our main theorem can be quickly drawn and the  {arguments can be}
followed more easily.  For this reason we have chosen to dedicate a section to this special case.  

The paper is organised as follows.  Section~\ref{s: Case BH} contains a description 
of the families of sets in the hyperbolic upper half plane that are relevant for the proof of our main theorem.
Section~\ref{s: Damek--Ricci spaces} contains some background material concerning Damek--Ricci spaces, together with
the definition of the families of sets that will be involved in the proof of our main result, and some of the properties of 
the associated maximal operators. 
The proof of parts \rmi\ and \rmii\ of our main result is contained in Section~\ref{s: The upper half-plane}.
Section~\ref{s: The upper half plane iii} is dedicated to the proof of Theorem~\ref{t: main}~{\rmiii}.
An interesting variant of $\cN^\frb$ is considered in Section~\ref{s: Variants}.  

We use the ``variable constants convention'', and denote by~$C$ a constant whose actual value may vary from place to place 
and which might depend on factors quantified before its occurrence, but not on factors quantified after. 

The expression
$$
A(t) \asymp B(t) \quant t \in {\bD},
$$
where ${\bD} $ is some subset of the domains of $A$ and of $B$, means that there exist (positive) constants $C$ and $C'$ such that
$$
C\,  A(t) \leq B(t) \leq C'\, A(t) \quant t \in {\bD};
$$
$C$ and $C'$ may depend on any quantifiers written \textit{before} the displayed formula.

\section{Relevant sets and reduction of the problem on $\BH$} \label{s: Case BH}

The purpose of this section is to describe the families of sets in the
hyperbolic upper half plane that are relevant in the proof of our main result.
Most of the proofs of their properties will be omitted: we shall 
give full proofs of the corresponding results in Damek--Ricci spaces.  

It is well known that the distance $d$ associated to the Riemannian metric \eqref{f: hyp metric} is given by 
$$
d(z,w)
= 2 \, \arcsinh \, \frac{\mod{z-w}}{2 \, \sqrt {\Im z\, \Im w}}
\quant z, w \in \BH.  
$$
We denote by $B_R(z)$ the (open) ball with centre $z$ and radius~$R$ with respect to the distance $d$, and by 
$\fB$ the family of all these balls.  

We recall that the geodesics in $\BH$ are either straight vertical rays orthogonal to the $x$-axis, or Euclidean upper 
half circles with centre on the $x$-axis.  { Given $z$ in $\BH$, we set $h(z) := \Im z$; we call $h(z)$
the \textit{height} of $z$.  Note that the Busemann function associated to the upward vertical geodesic starting at $i$
is the function $\log h$.  We observe that $\log h(z)$ is equal to $d(z,i\Im z)$ if $h(z)\geq 1$, and to $-d(z,i\Im z)$
if $0 < h(z) < 1$.}
For each $u>0$ we set 
{ 
$$
\cH_u 
:= \big\{z\in \BH: h(z) = u \big\}
\quad\hbox{and}\quad
\Pi_u 
:= \big\{z\in \BH: h(z) > u \big\}.
$$
}
Notice that $\cH_u$ is a horocycle with respect to the point at infinity in the vertical direction.

For $z$ in $\BH$, the \textit{special half plane through} $z$ is the region $T(z)$ in $\BH$  {lying}
below the geodesic through $z$ which is tangent at $z$ to the horocycle $\cH_{\Im z}$.  Such geodesic is
the upper half of 
the Euclidean circle with centre $(\Re z,0)$ and radius $\Im z$.  Another way to describe $T(z)$ is the following:  {
\begin{equation} \label{f: triangoloid}
	T(z)
	:= \big\{w\in \BH: \bigmod{w-\Re z} < \Im z\big\},
\end{equation}
where $\mod{\cdot}$ denotes the complex modulus.  }

For each positive number $R$ and for each point $z$ in $\BH$ we consider 
$$
b_R(z) := B_R(z) \cap T(z),
$$
which we call the Riemannian (lower) \textit{half ball with centre $z$ and radius~$R$}.  
We observe that $b_R(z)$ is a half ball in the geometric sense, because the intersection of the ball $B_R(z)$ with 
the geodesic through $z$ tangent to the horocycle $\cH_{\Im z}$ is indeed a diameter of $B_R(z)$.  
Set 
$$
\frb:= \big\{b_R(z): z\in \BH,\, R>0\big\}.
$$
It is well known that $SL(2,\BR)$ acts isometrically on $\BH$ via fractional linear transformations.  
Specifically, the matrix $M = \ds\matrice{\al}{\be}{\ga}{\de}$ in $SL(2,\BR)$ acts on the point $w$ in $\BH$ by 
$$
M\cdot w
:= \frac{\al w+\be}{\ga w + \de}.
$$
In particular, for each $z$ in $\BH$, we consider the matrix
\begin{equation} \label{f: Mz}
M_z 
:= \matrice{\sqrt{\Im z}}{\Re z/\sqrt{\Im z}}{0}{1/\sqrt{\Im z}},
\end{equation}
and note that
$$
M_z \cdot w
= w \Im z+\Re z
\quant w \in \BH.
$$
Clearly this action preserves the hyperbolic distance and the Riemannian area.

We now collect a few formulae which will be useful later.  We have
$$
d(iy_1,iy_2) = \log (y_2/y_1)
\qquad 0<y_1<y_2.
$$
The geodesic ball $B_R(z)$ agrees with the Euclidean disc with centre 
$$
\Re z + i(\Im z) \cosh R
\qquad\hbox{and radius}\qquad (\Im z) \sinh R.
$$
Thus, the boundary of $B_R(z)$ is the circle $C_R(z)$ consisting of all points $x+iy$ such that
$$
(x-\Re z)^2+\big(y-(\Im z)\,\cosh R\big)^2 = (\Im z)^2\,(\sinh R)^2.
$$
It is straightforward to check that the point { on this circle with lowest height} is $\Re z + i\e^{-R} \, \Im z$ and
that the horocycle $\cH_{\e^{-R} \, \Im z}$ is tangent to $C_R(z)$ at this point. 

Note that $C_R(z)$ intersects the boundary of $T(z)$ at the points
\begin{equation} \label{f: qzRm}
q_{z,R}^-
:= \Re z- (\Im z)\tanh R +i \,\frac{\Im z}{\cosh R}
\end{equation}
and
\begin{equation} \label{f: qzRp}
q_{z,R}^+
:= \Re z+ (\Im z)\tanh R +i \,\frac{\Im z}{\cosh R}.  
\end{equation}
Notice that $b_R(z)$ contains the Euclidean segment joining these two points. 

For $R>0$ and $z$ in $\BH$, we set
\begin{equation} \label{f: triangoloid II}
	T_R(z)
	:= T(z) \cap {\Pi}_{\e^{-R} \Im z}.
\end{equation}
We call $T_R(z)$ the \textit{trigonon with vertex $z$ and height~$R$}.  
Thus, $T_R(z)$ is the subset of $T(z)$ of all points that lie on or above the horocycle $\cH_{\e^{-R}\Im z}$.  
Note that $T(z) \cap \cH_{\e^{-R} \Im z}$ is the (Euclidean) segment joining the points 
$$
p_{z,R}^- := \Re z - (\Im z) \, \sqrt{1-\e^{-2R}} + i\e^{-R}\Im z
$$ 
and 
$$
p_{z,R}^+ := \Re z + (\Im z) \, \sqrt{1-\e^{-2R}}+i \e^{-R}\Im z.
$$
Set 
$$
\fT_\infty := \big\{T_R(z): z\in \BH,\, R>1\big\}.  
$$
Finally we consider the infinite (open) rectangle 
\begin{equation} \label{f: def QRz}
Q_R(z)
:= (\Re z-\Im z, \Re z+\Im z) \times (\e^{-R} \, \Im z, \infty).
\end{equation}
Clearly $Q_R(z)$ contains $b_R(z)$.  Set 
$$
\fQ_\infty := \big\{Q_R(z): z\in \BH,\, R>1\big\}.
$$
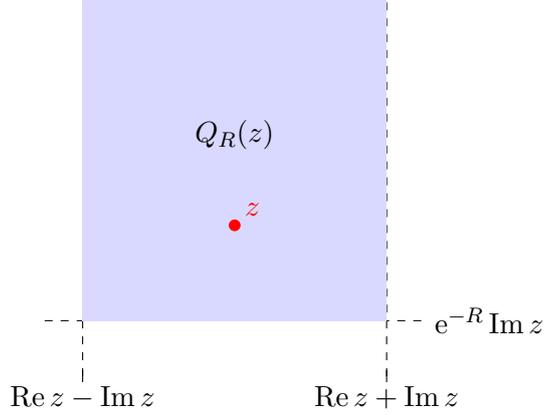
\begin{figure}[ht]
\centering
\begin{tikzpicture}[scale=1.0]

\def\x{1.5}
\def\y{2}

\def\R{1}

\pgfmathsetmacro{\ymin}{exp(-\R)*\y}


\draw[dashed] ({\x-\y},0) -- ({\x-\y},5);
\draw[dashed] ({\x+\y},0) -- ({\x+\y},5);

\draw ({\x-\y},0.08) -- ({\x-\y},-0.08);
\node[below] at ({\x-\y},0) {$\Re z - \Im z$};

\draw ({\x+\y},0.08) -- ({\x+\y},-0.08);
\node[below] at ({\x+\y},0) {$\Re z + \Im z$};

\draw[dashed] (-1,{\ymin}) -- (4,{\ymin});
\node[right] at (4,{\ymin}) {$\e^{-R}\Im z$};

\fill[blue!15]
    ({\x-\y},{\ymin}) rectangle ({\x+\y},5);

\filldraw[red] (\x,\y) circle (2pt)
    node[above right] {$z$};

\node at (\x,3.2) {$Q_R(z)$};

\end{tikzpicture}
	\caption{The infinite rectangle $Q_R(z)$}\label{rectangles}

	\end{figure}
Straightforward computations prove the following.

\begin{proposition} \label{p: properties sets}
	For any $R>1$ and $z$ in $\BH$ we have that
	\begin{equation} \label{f: area of triangoloid}
	\left. \begin{aligned}
		\bigmod{B_R(z)} & = 4\pi\, \big(\sinh(R/2)\big)^2\\
		\bigmod{b_R(z)} & = (1/2)\, \bigmod{B_R(z)} \\
		\bigmod{T_R(z)} & = 2\, \e^R\, \sqrt{1-\e^{-2R}} \, - 2\, \arccos \big(\e^{-R}\big) \\
		\bigmod{Q_R(z)} & = 2\, \e^R
	\end{aligned}
		\right\}
		\asymp \e^R. 
	\end{equation}
\end{proposition}

For any $Q$ in $\fQ_\infty$ we denote by ${ \be} (Q)$ its base.  For each nonnegative integer~$k$ we define $\be_k(Q)$ by
$$
\be_k(Q)
:= \big\{w\in Q: k< d\big(w, { \be} (Q)\big) \leq k+1 \big\}.
$$
Clearly $\ds Q = \bigcup_{k=0}^\infty \, \be_k(Q)$.  
Notice that ${ \be} (Q)$ is contained in the ``lower boundary'' of $\be_0 (Q)$. 
A straightforward computation shows that
$$
\ds \mod{Q} \asymp \bigmod{\be_0(Q)}
\quant Q \in \fQ_\infty.
$$

\subsection{Reduction of the problem}
We call an infinite rectangle $Q_R(z)$ \textit{admissible} if $\Im z$ is of the form $\e^j$ with $j$ in $\BZ$,
and~$R$ is a nonnegative integer.  Observe that the ``base" of $Q_R(z)$ is contained in the horocycle~$\cH_{\e^{j-R}}$.  
Set
$$
\fQ_\infty'
:= \big\{\hbox{admissible infinite rectangles in $\fQ_\infty$} \big\}.
$$
For some of the estimates we shall prove in the sequel, it will be more convenient from a geometric point of view to consider
uncentred maximal functions with respect to $\fT_\infty$ and $\fQ_\infty'$ rather than $\frb_\infty$.  
Some useful relations between~$\cN^{\frb_\infty}$ and either $\cN^{\fT_\infty}$ or $\cN^{\fQ_\infty'}$ are described in 
Propositions~\ref{p: inclusions T and b} and~\ref{p: admissible}, respectively.

\begin{proposition} \label{p: inclusions T and b}
	The following hold:
	\begin{enumerate}
		\item[\itemno1]
		$
		T_{R-\log 2}(z) \subset b_R(z) \subset T_R(z) \subset b_{R+\log 2}(z) 
		$
		for every $z$ in $\BH$ and for every $R\geq 1$;
		\item[\itemno2]
		there exist positive constants $c$ and $C$ such that 
		$$
		c\, \cN^{\fT_\infty}\! f
		\leq \cN^{\fb_\infty}\! f 
		\leq C\, \cN^{\fT_\infty}\! f.   
		$$
	\end{enumerate}
\end{proposition}

\begin{proposition} \label{p: admissible}
	There exists a constant $C$ such that the following hold:
	\begin{enumerate}
		\item[\itemno1]
		for every $z$ in $\BH$ and $R$ in $(1,\infty)$ there exists an admissible infinite rectangle $Q$ 
		containing $Q_R(z)$ satisfying
		$$
			\frac{\mod{Q}}{\mod{Q_R(z)}}
			\leq C;   
		$$
		\item[\itemno2]
		for every measurable function $f$ the following inequality holds
		\begin{equation} \label{f: inequality max adm}
			\cN^{\frb_\infty}\! f 
			\leq C\,\, \cN^{\fQ_\infty'}\! f.
		\end{equation}
	\end{enumerate}
\end{proposition}

A direct consequence of these propositions is that it suffices to prove Theorem~\ref{t: main}~\rmi-\rmii\ with $\cN^{\fQ_\infty'}$
in place of $\cN^{\fb_\infty}$, and Theorem~\ref{t: main}~\rmiii\ with~$\cN^{\fT_\infty}$ in place of $\cN^{\fb_\infty}$. 

\section{Damek--Ricci spaces} \label{s: Damek--Ricci spaces}
Consider a real Lie algebra $\frn$, equipped with a scalar product 
$\prodo{\cdot}{\cdot}$, that can be written as an orthogonal direct sum, thus
$$
\frn
= \frv \oplus \frz, 
$$
where $[\frv, \frv] \subseteq \frz$ and $[\frv,\frz] = 0 = [\frz,\frz]$.  If either $\frv$ or $\frz$ is reduced to $0$,
then we say that $\frn$ is \textit{degenerate}.  Denote by $p$ and $q$ the dimensions of~$\frv$ and $\frz$, respectively.  
The \textit{homogeneous dimension} $\nu$ of $\frn$ is then defined as
$$
\nu
:= \frac{p}{2} + q.  
$$
For each $Z$ in~$\frz$, consider the
linear map $J_Z$ on $\frv$, defined by
\begin{equation} \label{f: def JZ}
\prodo{J_ZX}{X'}
:= \prodo{Z}{[X,X']}
\quant X, X' \in \frv.  
\end{equation}
We shall assume throughout that $\frn$ is $H$-type, i.e.
\begin{equation} \label{f: def H type}
J_Z^2 = -\mod{Z}^2\, I,
\end{equation}
where $I$ denotes the identity mapping on $\frv$, and $\mod{Z}^2$ is short for $\prodo{Z}{Z}$.  There is an extensive literature
on $H$-type groups, part of which can be found in the bibliographic references of \cite{CDKR1, CDKR2}.  We refer to these 
two papers for any unexplained notation concerning $H$-type groups and their harmonic extensions (defined below).  
It is also well known that 
\begin{equation} \label{f: polarisation}
\prodo{J_ZX}{J_ZY}
= \mod{Z}^2 \, \prodo{X}{Y}
\quant X,Y \in \frv \quant Z\in \frz.
\end{equation}
We shall consider $H$-type groups $N$, i.e. connected, simply connected real Lie groups whose algebra is $H$-type.  The
exponential mapping is a bijection of $\frn$ onto $N$.  Thus, we consider the parametrisation of $N$ that assigns to an
element $\exp(X+Z)$ in $N$ (here $X$ and $Z$ belong to $\frv$ and $\frz$, respectively) the coordinates $(X,Z)$. 
By the Baker--Campbell--Hausdorff formula, in these coordinates, the group law takes the form
\begin{equation} \label{f: group law N}
(X,Z) \cdot (X',Z')
= \big(X+X',Z+Z'+ (1/2) \, [X,X']\big),
\end{equation}
where $(X,Z)$ and $(X',Z')$ are in $N$. 

We consider the left-invariant Riemannian metric on $N$ that agrees at~$e_0$ (the neutral element in $N$) with the chosen 
scalar product on $\frn$.  The corresponding unit speed geodesics in $N$ through $e_0$ are described in 
\cite[Proposition~2.1]{CDKR1}.  

We shall also need the gauge $\cG$ on $N$, defined by 
\begin{equation} \label{f: group law N}
	\cG(X,Z)
	:= \Big(\frac{\mod{X}^4}{16} + \mod{Z}^2\Big)^{1/4}
	\quant (X,Z) \in \frn.  
\end{equation}
If $n$ denotes the element $\exp(X+Z)$, then we also write $\cG(n)$ in place of $\cG(X,Z)$.
By a result of J. Cygan \cite{C}, the gauge $\cG$ is a homogeneous norm, and we denote by $d_N$ the corresponding 
left-invariant distance on $N$, defined by 
\begin{equation}\label{eq:distN}
d_N(n, n')
	:=\cG \big(n^{-1}n'\big).
\end{equation}
For each positive number $R$, we denote by $\cB_R(n_0)$ the ball in $N$ with centre~$n_0$ and radius $R$ with
respect to the distance $d_N$.

The Lebesgue measure $\wrt n$ is a translation-invariant measure on $N$ and the metric measure space $(N, d_N, \wrt n)$ is 
homogeneous in the sense of Coifman and Weiss.  In particular, the Lebesgue measure is doubling on $N$. 

We denote by $A$ the multiplicative group of the positive reals.  Its Lie algebra $\fra$ can be identified with
$\BR$ via the map $t \mapsto tH$, where $H$ is a fixed element in $\fra$.  We define the inner product on $\fra$ for
which $\mod{H} = 1$.  We consider the semidirect product group $NA$, defined as follows.  Denote by $\frs$
the Lie algebra $\frn\oplus\fra$, where the Lie bracket   {is} determined by linearity and the requirement that
$$
[H, X] = \frac{1}{2} \, X
\qquad\hbox{and}\qquad
[H, Z] = Z
$$
for all $X$ in $\frv$ and $Z$ in $\frz$.   

We denote by $(X,Z,t)$ the element $X+Z+tH$ in $\frs$: here $X$ belongs to $\frv$, $Z$ belongs to $\frz$,
and $t$ is a real number.  
We extend the inner products on $\frn$ and $\fra$ to an inner product on $\frs$, by the rule
$$
\prodo{(X,Z,t)}{(X',Z',t')}_{\frs}
:= \prodo{(X,Z)}{(X',Z')}_{\frn} + tt'.
$$
In particular 
$$
\prodo{(X,Z,t)}{(X,Z,t)}_{\frs}
= \mod{X}^2 + \mod{Z}^2 + t^2.
$$
For every positive number $R$ set
\begin{equation} \label{f: B1 frs}
	B_1(\frs) 
	:= \big\{(X,Z,t) \in \frs: \mod{X}^2+\mod{Z}^2 + t^2 < 1 \big\},
\end{equation}
and
\begin{equation} \label{f: S1 frs}
	S_1(\frs) 
	:= \big\{(X,Z,t) \in \frs: \mod{X}^2+\mod{Z}^2 + t^2 = 1 \big\}.
\end{equation}

We consider the left-invariant metric on $\exp(\frs)$ which agrees with the metric defined above on $\frs$, 
when we view $\frs$ as the tangent space to $\exp(\frs)$ at the identity.  The associated left-invariant distance on $S$ 
will be denoted by $d$.

A convenient parametrisation of $\exp(\frs)$ is obtained by denoting the element $\exp(X+Z) \, \exp\big((\log s)H\big)$ of $NA$  
 {simply by $(X,Z,s)$,} where $X$ is in~$\frv$,~$Z$ is in~$\frz$ and $s$ is a positive real number.  
Then it is straightforward to check that
\begin{equation} \label{f: group law NA}
	(X,Z,s) \cdot (X',Z',s')
	= \big(X+ \sqrt{s}\, X', Z+ s\, Z' + (\sqrt s/2) \, [X,X'], ss'\big),
\end{equation}
that $(0,0,1)$, which we denote by $e$, is the neutral element in $NA$ and that
$$
	(X,Z,s)^{-1}
	= \big(-(1/\sqrt s) \, X, -(1/s)\, Z , 1/s\big).
$$
for every $(X,Z,s)$ and $(X',Z',s')$ in $\exp(\frs)$.  

Note that $N$ (identified with the set of all elements of the form $(X,Z,1)$) is normal in $NA$, because
$$
(X',Z',s)^{-1}\, (X,Z,1) \, (X',Z',s) 
= \Big(\frac{1}{\sqrt s} \, X, \frac{1}{s} \, Z, 1\Big). 
$$
In particular, this formula says that $A$ acts on $N$ by anisotropic dilations.
We normalise the left-invariant measure on $NA$ as follows
\begin{equation} \label{f: Haar NA}
\wrt (X,Z,s) 
:= s^{-\nu-1} \, \wrt X \wrt Z \wrt s.
\end{equation}
It is well known (see e.g. \cite[Eq. (2.18)]{ADY}), that the geodesic distance $r(X,Z,s)$ of the point $(X,Z,s)$ 
from the origin satisfies
\[ 
\cosh^2\!\Big(\frac{r(X, Z, s)}{2}\Big)
=
\Big( \cosh(\log \sqrt s) + \frac{1}{8} \,\frac{|X|^2}{\sqrt s} \Big)^2 + \frac{1}{4} \,\frac{|Z|^2}{s}.
\]
By the left-invariance of the metric $d$ and the group law \eqref{f: group law NA}, it is straightforward to check that
\begin{equation}\label{eq:vertical_dist}
	\begin{aligned}
    	d\big((X, Z, s), (X, Z, s') \big)
		& = r\big((X, Z, s')^{-1} (X, Z, s)\big) \\
		& = r\big((0,0,s/s')\big) \\  
		& = \left|\log\frac{s}{s'}\right| \\
	\end{aligned}
\end{equation}
and, similarly, that if $n$ denotes the point $\exp(X+Z)$ of $N$, then 
\begin{equation}\label{eq:N_incr_dist}
	d(na, a')
	=d\big((0,0,1/a')\cdot (X, Z,a){ , (0,0,1)}\big)\geq \log\,\Bigmod{\frac{a}{a'}}
	=d(a, a').
\end{equation}

For each {vector $v := (X,Z,t)$ in $S_1(\frs)$}, denote by $\ga_{v}$ the geodesic in $NA$ such that 
$$
\ga_{v} (0) = e
\qquad\hbox{and}\qquad
\dot\ga_{v} (0) = { v}.  
$$
We recall the explicit formula for $\ga_v$ established in \cite[Proposition~2.2]{CDKR1}: 
\begin{equation} \label{f: geodesics S0}
	\ga_v(\tau)
	:= \Big(\frac{2\te(1-t\te)}{\chi}\, X - \frac{2\te^2}{\chi} \, J_ZX, \frac{2\te}{\chi}\, Z, \frac{1-\te^2}{\chi}\Big),
\end{equation}
where $\te := \tanh(\tau/2)$, $\tau \geq 0$,  and $\chi := (1-t\te)^2 + \mod{Z}^2 \, \te^2$.  In the ``ball model'' of $S$,
these geodesics correspond to straight lines through the origin: see \cite{CDKR1}.  
Set 
$$
T(e)
:= \big\{\ga_v(\tau): v \in { B_1(\frs)}, \,\, t{ <} 0, \, \, \tau \in { (}0,\infty) \big\}.  
$$ 
Clearly $T(e)$ (which is the analogue in $S$ of the special half plane $T(i)$ in $\BH$ introduced in \eqref{f: triangoloid})
is the set of all points on unit speed geodesics issuing from $e$ with tangent pointing ``downward''. 
In the ``ball model'' of $S$, the set $T(e)$ corresponds to a half of a ball centred at the origin. 

From \eqref{f: geodesics S0} we see that if $v = (0,0,-1)$, then the geodesic $\ga_v$ has the form
$$
\ga_v(\tau)
= \Big(0,0, \frac{1- \big(\tanh (\tau/2)\big)^2}{\big(1 + \tanh (\tau/2)\big)^2} \Big)
= \big(0,0, \e^{-\tau} \big). 
$$
Since $\ga_v$ is a unit speed geodesic, we deduce that
\begin{equation} \label{f: dist from e}
d\big(e,(0,0,\e^{-\tau}) \big)
= \tau,   
\end{equation}
where $d$ denotes the Riemannian distance on $NA$.  

Define the \textit{height function} $h$ on $NA$ by the rule
$$
h(na)
:= a 
\quant na \in NA.
$$
It is straightforward to check that the height of a point is preserved under left translations by elements in $NA$ of the 
form~$n\cdot 1$.  It is known that $h$ is the Busemann function associated to ``the point at infinity'' 
in the vertical direction.

For each $u>0$ we set 
$$
\cH_{u}
:= \big\{{ nu \in NA: n\in N} \big\}.
$$
Notice that $\cH_u$ is a horocycle with respect to the point at infinity in the vertical direction.  
Then the orbit of the point $e_0 u$ in $NA$ under the (normal) subgroup $N$ of $NA$ is just $\cH_{{ u}}$.
We also need  
$$
\Pi_u 
:= \big\{{ na}\in NA: { n \in N, \, a} > u \big\}.   
$$

\begin{remark} \label{rem: ball boundary I}
	\rm{
		Consider the geodesics $\ga_v$ in \eqref{f: geodesics S0},
		where $v := (X,Z,t)$, $t\leq 0$, and $\mod{X}^2 + \mod{Z}^2  + t^2= 1$.  
		We \textit{claim} that 
		$$
			\e^{-R}
			\leq h\big(\ga_v(R)\big) 
			\leq \frac{1}{\big(\cosh (R/2)\big)^2}.  
		$$
		Indeed, observe that $\chi = 1-2t\te + (t^2+\mod{Z}^2) \, \te^2$, and recall that $t$ is nonpositive.  
		Therefore $\chi$ is always greater than or equal to $1$.  Notice that $\chi = 1$ when
		$t=0=\mod{Z}$ and $\mod{X}=1$ (a choice which satisfies all the constraints).  This proves that
		$h\big(\ga_v(R)\big) \leq 1-\te^2 = 1/\big(\cosh (R/2)\big)^2$.   

		As to the left hand inequality, observe that $t^2+\mod{Z}^2 = 1- \mod{X}^2$, which reaches its maximum
		when $X=0$.  The summand $1-2t\te$ reaches its maximum when $t$ is $-1$.  So, if $\mod{X} = 0 = \mod{Z}$
		and $t=-1$, we have that $\chi = (1+\te)^2$.  Then
		$$
			h\big(\ga_v(R)\big)
			\geq \frac{1-\te^2}{(1+\te)^2} 
			= \frac{1-\te}{1+\te} 
			= \e^{-R},
		$$
		and the claim is proved. 

		In particular, notice that the height of the point $\ga_v(R)$ is not constant as~$v$ varies in
		the set of all points of the form $(X,Z,0)$, where $\mod{X}^2+\mod{Z}^2 = 1$.
	}
\end{remark}

For each $h$ in $(0,1)$, define 
\begin{equation} \label{f: al1 al2} 
	\al_1(h) := (1-h)^{1/2}
	\quad\hbox{and}\quad 
	\al_2(h) := (1-h^2)^{1/4}.
\end{equation}

\begin{proposition} \label{p: shadow on the boundary}
	Suppose that $h$ is in $(0,1)$.  Then every point ${ nh}$ in $\partial\big(T(e)\big) \cap \cH_{h}$ satisfies
	$$
	{ \al_1(h)
	\leq \cG(n)
	\leq \al_2(h).}
	$$
\end{proposition}

\begin{proof}
	{ By \eqref{f: geodesics S0}, if $v := (X,Z,0)$ is a ``horizontal'' vector in $\frs$ 
	such that $\mod{X}^2 + \mod{Z}^2 = 1$, then the points of the form}
	\begin{equation} \label{f: geodesics S}
		\ga_v(\tau)
		= \Big(\frac{2\te}{\chi}\, X - \frac{2\te^2}{\chi} \, J_ZX, \frac{2\te}{\chi}\, Z, \frac{1-\te^2}{\chi}\Big),
	\end{equation}
	where $\te = \tanh(\tau/2)$ and $\chi = 1+ \mod{Z}^2 \, \te^2$, are all the points in 
	$\partial \big(T(e)\big)$.  If~{ $\ga_v(\tau)$} belongs to $\cH_h$, then 
	$$
		\frac{1-\te^2}{\chi} = h,
	$$
	which yields
	$$
		\te^2 = \frac{1-h}{1+ h\mod{Z}^2}.
	$$
	Denote by $\tau_h$ the unique value of $\tau$ which solves this equation (given $h$ and $Z$).  By substituting
	this value of $\te$ in the definition of $\chi$, we see that 
	$$
		\chi 
		= 1 + \mod{Z}^2 \, \frac{1-h}{1+ h\mod{Z}^2}
		= \frac{1+ \mod{Z}^2}{1+h\mod{Z}^2}.  
	$$
	In order to get cleaner formulae, we set 
	$$
		a := \frac{\te}{\chi}
		\qquad\hbox{and}\qquad 
		b := \frac{\te^2}{\chi}.
	$$
	Of course $a$ and $b$ depend on $h$, as we think of $\te$ and $\chi$ evaluated at the point $\tau_h$.  However, 
	for simplicity, we do not stress this dependence in our notation.  Summing up, we have 
	$$
		\ga_v(\tau_h)
		:= \big(2a X - 2b  J_ZX, 2a Z, h\big).  
	$$
	Next, we observe that  
	\begin{equation} \label{f: cN I}
	\begin{aligned}
		\cG \big(2a X - 2b  J_ZX, 2a Z \big)^4
		& = \frac{1}{16} \, \bigmod{2a X - 2b  J_ZX}^4 + 4a^2\mod{Z}^2 \\  
		& = \big(a^2 \mod{X}^2 + b^2  \mod{J_ZX}^2\big)^2 + 4a^2\mod{Z}^2 \\  
		& = \mod{X}^4\, \big(a^2+ b^2  \mod{Z}^2\big)^2 + 4a^2\mod{Z}^2:
	\end{aligned}
	\end{equation}
	we have used that $\prodo{X}{J_ZX} = 0$ (see \eqref{f: def JZ}) and that $\mod{J_ZX}^2 = \mod{X}^2\, \mod{Z}^2$ 
	(see \eqref{f: polarisation}).  Now, recall that $\ga_v$ is a unit speed geodesic, so that  
	$$
	\mod{X}^4 
	= \mod{X}^2 \, \mod{X}^2
	= \big(1-\mod{Z}^2\big)^2.  
	$$
	Moreover, from the definitions of $a$ and $b$, we see that 
	$$
		a^2 
		= \frac{\te^2}{\chi^2}
		= \frac{(1-h)\, \big(1+h\mod{Z}^2\big)}{\big(1+\mod{Z}^2\big)^2}
		\quad\hbox{and} \quad
		b^2
		= \frac{(1-h)^2}{\big(1+ \mod{Z}^2\big)^2}.  
	$$
	We use these formulae and \eqref{f: cN I}, and obtain that 
	$$
	\begin{aligned}
		& \cG \big(2a X - 2b  J_ZX, 2a Z \big)^4 \\
		& = \big(1-\mod{Z}^2\big)^2 \, \Big[\frac{(1-h)\, \big(1+h\mod{Z}^2\big)}{\big(1+\mod{Z}^2\big)^2}
			+ \frac{(1-h)^2}{\big(1+ \mod{Z}^2\big)^2} \, \mod{Z}^2  \Big]^2 \\
		&  \qquad + 4 \, \frac{(1-h)\, \big(1+h\mod{Z}^2\big)}{\big(1+\mod{Z}^2\big)^2} \, \mod{Z}^2
			\\
		& = \big(1-\mod{Z}^2\big)^2 \, \frac{(1-h)^2}{\big(1+\mod{Z}^2\big)^2}
			+ 4 \, \frac{(1-h)\, \big(1+h\mod{Z}^2\big)}{\big(1+\mod{Z}^2\big)^2} \, \mod{Z}^2
			\\
		& = (1-h)^2 + 4 \,(1-h)\, h\,  \frac{\mod{Z}^2}{1+\mod{Z}^2}.
	\end{aligned}
	$$
	Since $\mod{Z}^2/\big(1+\mod{Z}^2\big) \leq 1/2$, we can conclude that 
	$$
		(1-h)^2 
		\leq \cG \big(2a X - 2b  J_ZX, 2a Z \big)^4
		\leq (1-h) \, (1+h),
	$$
	as required.
\end{proof}

For each positive number $R$ and each $h$ { in $A$}, define  
{ $$
\cB_R(e_0h)
:= \cB_R(e_0) \cdot h
$$
(recall that $e_0$ denotes the neutral element in $N$): here $\cdot$ denotes the product in $NA$.}
Clearly $\cB_R(e_0h)$ is ``a copy at height $h$ in $S$'' of the (open) ball { $\cB_R(e_0)$ in $N$.}


For each positive number $R$, set
$$
T_R(e)
:= T(e) \cap \Pi_{\e^{-R}}.  
$$
For each $h$ in $(0,1)$, 
define 
the sets
$$
\ucT_R(e)
:= \bigcup_{h\in (\e^{-R},1)} \, {\cB_{\al{ _1}(h)}(e_0h)} 
$$
and
$$
\ocT_R(e)
:= \bigcup_{h\in (\e^{-R},1)} \, {\cB_{{ \al_2}(h)}(e_0h)},
$$
where $\al_1$ and $\al_2$ are defined in \eqref{f: al1 al2}.
Denote by 
$C$ the Cayley type transform from $S$ to { $B_1(\frs)$ (defined in \eqref{f: B1 frs})}.  
See \cite[formula (1.12)]{ADY} for more information about $C$.  Straightforward calculations show that 
$C$ maps $T(e)$ onto 
$$
	\big\{(X,Z,t) \in \frs: \mod{X}^2+\mod{Z}^2 + t^2 < 1\,\,\hbox{and}\,\, t<0 \big\} {.}
$$
As a consequence, 
\begin{equation} \label{f: ucT ocT}
\ucT_R(e)
\subset T_R(e) 
\subset \ocT_R(e).  
\end{equation}
Other sets of interest will be the Riemannian ball $B_R(e)$, and the Riemannian half ball $b_R(e)$, defined by
$$
b_R(e)
:= B_R(e) \cap T(e).  
$$
Finally, we consider cylinders of the form
$$
\cC_R(e)
:= { \cB_{1}(e_0) \times \big(\e^{-R},\infty\big).}
$$
Notice the inclusion $b_R(e) \subset \cC_R(e)$.  Indeed, this follows by the definition of $T (e)$ 
and the fact that the ``south pole'' of $B_R(e)$ is $(0,0,\e^{-R})$.  Furthermore Proposition~\ref{p: shadow on the boundary} 
implies that
$$
T_R(e)
\subset \cC_R(e)
\quant R > 1.
$$

Owing to translation-invariance of the setting, half balls and cylinders centred at any point 
$n_0a_0$ in $NA$ are defined by translating the corresponding sets at the origin, namely
$$
b_R(n_0a_0):=n_0a_0\, b_R(e),
$$
and
$$
\cC_R(n_0a_0):=n_0 a_0 \,\cC_R(e){ .}
$$
{ Observe that a point in $\cC_R(n_0a_0)$ is of the form 
$$
n_0a_0na
= n_0 n^{a_0} a_0 a,
$$
where $na$ belongs to $\cC_R(e)$ and $n^{a_0}$ is short for $a_0 n a_0^{-1}$.  Notice that, by the formula
above \eqref{f: Haar NA}, 
$\cG\big(  (n^{a_0})^{-1} \big) = \sqrt{a_0} \,\, \cG(n)$.   Hence
$$
d_N(n_0 n^{a_0}, n_0)
= \cG\big( (n_0 n^{a_0})^{-1} n_0 \big)
= \cG\big(  (n^{a_0})^{-1} \big)
= \sqrt{a_0} \,\, \cG(n)
< \sqrt{a_0}.  
$$
{ Also $a_0a>a_0 \, \e^{-R}$, so that}
\begin{equation}\label{eq:cylinder_n0a0}
\cC_R(n_0a_0)
= \cB_{\sqrt{a_0}}(n_0) \times \big(a_0\, \e^{-R},\infty\big).
\end{equation}
}
We set $T_R(n_0a_0):=n_0a_0\, T_R(e)$. 

We denote by $\fb$, $\fT$ and $\fC$ the families of all half balls, trigona and cylinders,
and by $\fb_\infty$, $\fT_\infty$ and $\fC_\infty$ the subfamilies thereof consisting of all elements associated
to parameters $R>1$.  

\begin{definition} \label{def: cC admissible}
	\rm{We call a cylinder $\cC_R(n_0a_0)$ \emph{admissible} if $\log a_0=j$ { is an integer}, and the radius~$R$ 
	is a nonnegative integer.   Observe that by \eqref{eq:cylinder_n0a0}, the ``base" of $\cC_R(n_0a_0)$ is contained 
	in the horocycle~$\cH_{\e^{j-R}}$.  We denote by~$\fC_\infty'$ the family of all admissible cylinders in $\fC_\infty$.
	}
\end{definition}

We need the following preliminary result. 

\begin{proposition} \label{p: properties setsNA}
	We have that 
	\begin{equation} \label{f: area of triangoloidNA}
	\left. \begin{aligned}
		&\bigmod{B_R(x)}  
		\\
		&\bigmod{b_R(x)}  
		\\
		&\bigmod{T_R(x)}  
		\\
		&\bigmod{\cC_R(x)}  
	\end{aligned}
		\right\}
		\asymp \e^{\nu R}
		\quant R>1 \quant x \in S. 
	\end{equation}
\end{proposition}

\begin{proof}
	Since the Riemannian measure is translation-invariant, it suffices to consider sets ``centred'' at $e$.  

	The result for the Riemannian ball $B_R(e)$ is well known (see, e.g., \cite[Eq. (1.18)]{ADY}).
	Straightforward calculations show that the Cayley transform~$C$ maps $b_R(e)$ onto 
	$$
		\big\{(X,Z,t) \in \frs: \mod{X}^2+\mod{Z}^2 + t^2 < \big(\tanh (R/2)\big)^2\,\,\hbox{and}\,\, t<0 \big\},
	$$
	Since the density of the measure on $B(\frs)$ is radial (see \cite[Formula~(1.16)]{ADY}), it is clear that
	the measure of $C\big(b_R(e)\big)$ is one half of the measure of $C\big(B_R(e)\big)$.  Then $|b_R(e)| = (1/2) \, |B_R(e)|$,
	because $C$ is an isometry.  

	For cylinders, we have
	\[
		|\cC_R(e)|
		= \int_{\mathcal{B}_1(e_0)}\!\!\wrt n\, \int_{\e^{-R}}^{\infty} a^{-\nu -1}\, \wrt a 
		= \frac{\om}{\nu}\, \e^{\nu R},
	\]
	{ where $\om$ denotes the measure of $\cB_1(e_0)$ (with respect to the measure on~$N$).} 

  	The inclusions \eqref{f: ucT ocT} yield
	$$
		\bigmod{\ucT_R(e)}
		\leq \bigmod{T_R(e)}
		\leq \bigmod{\ocT_R(e)}.  
	$$
	{ Recall that $\al_1(h) = (1-h)^{1/2}$ and $\al_2(h) = (1-h^2)^{1/4}$, and observe that}
	$$
	\begin{aligned}
		|\ucT_R(e)|
		& =     { \int_{\e^{-R}}^{1} \, \bigmod{\cB_{\al_1(h)}(e_0)} \, \frac{\wrt h}{h^{\nu+1}} } \\
		& =     \om\, \int_{\e^{-R}}^{1}\frac{\al_1(h)^\nu}{h^{\nu+1}} \wrt h \\
		& \geq  \om\, (1-\e^{ {-1/2}})^{\nu/2}  \, \int_{\e^{-R}}^{\e^{-1/2}}\frac{1}{h^{\nu+1}} \wrt h \\
		& =     \om\, (1-\e^{ {-1/2}})^{\nu/2}\, \,\frac{\e^{\nu R}-\e^{\nu/2}}{\nu}
		\quant R>1.
	\end{aligned}
	$$
	Similarly, 
	$$
	\begin{aligned}
		|\ocT_R(e)|
			& =     \int_{\e^{-R}}^{1} \, { \bigmod{\cB_{\al_2(h)}(e_0)} \,\frac{\wrt h}{h^{\nu+1}}} \\
			& =     \om \, \int_{\e^{-R}}^{1}\frac{{ \al_2}(h)^\nu}{h^{\nu+1}} \wrt h \\
			& \leq  \om \, \frac{\e^{\nu R}-1}{\nu}
			\quant R>1;
	\end{aligned}
	$$
	we have used the fact that $\al_2(h) < 1$ in the inequality above.  Therefore
	\begin{equation} \label{f: est volume trigonon}
		\bigmod{T_R(e)}
		\asymp \e^{\nu R} \quant R>1,
	\end{equation}
	as required. 
\end{proof}

For further reference, note that the proof of Proposition~\ref{p: properties setsNA} implies that 
\begin{equation} \label{f: trigonon more}
	c_\nu \, \e^{\nu R} 
	\leq \bigmod{T_R (e)}
	\leq C_\nu \, \e^{\nu R} 
	\quant R>1,
\end{equation}
where 
$$
{ c_\nu
:= \frac{\om}{\nu} \, \big(1-\e^{-1/2}\big)^{\nu/2}\, (1-\e^{-\nu/2})}
\quad\hbox{and}\quad
C_\nu
:= { \frac{\om}{\nu}}.
$$

\subsection{Reduction of the problem}

In this subsection we compare $\cN^{\frb_\infty}$ with $\cN^{\fT_\infty}$ 
and with $\cN^{\fC_\infty'}$.  Proposition~\ref{p: inclusions T and bNA}~\rmii\ will 
be used in the proof of Theorem~\ref{t: main}~\rmi-\rmii, and Proposition~\ref{p: admissibleNA} will play a role in the proof
of Theorem~\ref{t: main}~\rmiii. 

The next two propositions are the analogue for general Damek--Ricci spaces of Proposition~\ref{p: inclusions T and b} 
and Proposition~\ref{p: admissible} for the  {hyperbolic} upper half plane. 

\begin{proposition} \label{p: inclusions T and bNA}
	The following hold:
	\begin{enumerate}
		\item[\itemno1]
		$\ds
		T_{R-2\log 2}(z) \subset b_R(z) \subset T_R(z) \subset b_{R+2\log 2}(z) 
		$
		for every $z$ in $S$ and every  {$R > 2$};
		 
		\item[\itemno2]
		there exist positive constants $c$ and $C$ such that 
		$$
		c\, \cN^{\fT_\infty}\! f
		\leq \cN^{\fb_\infty}\! f 
		\leq C\, \cN^{\fT_\infty}\! f.   
		$$
	\end{enumerate}
\end{proposition}

\begin{proof}
	First we prove \rmi.  By translation-invariance of the metric, it suffices to prove the result for $z = e$.
	By definition of $b_R(e)$ and $T_R(e)$, it is clear that $b_R(z) \subset T_R(z)$.  The inclusion
	$T_R(z) \subset b_{R+2\log 2}(z)$ will follow from $T_{R-2\log 2}(z) \subset b_R(z)$ simply 
	by replacing $R$ with $R+2\log 2$.

	So it remains to prove that $T_{R-2\log 2}(z) \subset b_R(z)$.  By Remark~\ref{rem: ball boundary I}, each of the points in 
	$$
	\partial \big(b_R(e)\big) \setminus T(e)
	$$
	has height at   {most} $1/\big(\cosh(R/2)\big)^2$.  It is straightforward to check that 
	$$
		\frac{1}{\big(\cosh(R/2)\big)^2}
		\leq \e^{2\log 2-R};
	$$
	the required containment follows directly from this.  

	Next we prove \rmii.  On the one hand, \rmi\ and~ {Proposition~\ref{p: properties setsNA}}
	imply the following:
	$$
	\begin{aligned}
		\frac{1}{\mod{T_R(w)}} \, \int_{T_R(w)} \mod{f} \wrt \mu
		& \leq \frac{1}{\mod{b_R(w)}} \, \int_{b_{R+2\log 2}(w)} \mod{f} \wrt \mu \\
		& \leq \frac{C}{\mod{b_{R+2\log 2}(w)}} \, \int_{b_{R+2\log 2}(w)} \mod{f} \wrt \mu,
	\end{aligned}
	$$
	so that the left inequality in \rmii\ holds.  Similarly
	$$
	\begin{aligned}
		\frac{1}{\mod{b_R(w)}} \, \int_{b_{R}(w)} \mod{f} \wrt \mu 
		& \leq \frac{1}{\mod{T_{R-2\log 2}(w)}} \, \int_{T_{R}(w)} \mod{f} \wrt \mu \\
		& \leq \frac{C}{\mod{T_R(w)}} \, \int_{T_{R}(w)} \mod{f} \wrt \mu,
	\end{aligned}
	$$
	which proves the right inequality in \rmii.
\end{proof}

\begin{proposition} \label{p: admissibleNA}
	There exist positive constants $c$ and $C$ such that the following hold:
	\begin{enumerate}
		\item[\itemno1]
		for every $n_0 a_0$ in $S$ and $R$ in $(1,\infty)$ {,} the admissible cylinder
		$\cC_K(n_0a)$, where $K=\lceil R \rceil+2$ and $\log a=\lceil \log a_0 \rceil$,
		contains $\cC_R(n_0a_0)$ and satisfies
		\begin{equation} \label{f: rel cyl}
			c
			\leq \frac{\mod{\cC_K(n_0 a)}}{\mod{\cC_R(n_0 a_0)}}
			\leq C;
		\end{equation}
		\item[\itemno2]
		for every measurable function $f$ the following inequality holds
		\begin{equation} \label{f: inequality max adm}
			\cN^{\frb_\infty}\! f 
			\leq C\,\, \cN^{\fC_\infty'}\! f.
		\end{equation}
	\end{enumerate}
\end{proposition}

\begin{proof} 
	First we prove \rmi.  { The translation-invariance of the measure, 
	Proposition~\ref{p: properties setsNA} and the fact that $R+2\leq K\leq R+3$ imply that} 
	$$
	\mod{\cC_R(n_0 a_0)}
	= \mod{\cC_R(e)}
	\asymp \e^{\nu R}
	\asymp\mod{\cC_K(e)},
	$$  
	and \eqref{f: rel cyl} is proved. 

	It remains to prove that $\cC_K(n_0a)$ contains  {$\cC_R(n_0a_0)$.} Multiplying on the left by $(n_0a_0)^{-1}$, 
	we see that it is equivalent to having 
	\[
		\cC_R(e)\subset (0,0,a_0^{-1}a)\cdot\cC_K(e),
	\]
	which, by \eqref{eq:cylinder_n0a0}, is true if and only if
	\[
		 {a_0 \leq a,} \quad \e^{-R} > \frac{a}{a_0}\e^{-K}.
	\]
	The first inequality is immediate since $\log a=\lceil \log a_0 \rceil$, and in fact implies that  {
	$1\leq a/a_0 < \e$}; it follows that the second inequality is also true, since 
	 { $\e^{-R-3}< \e^{-K}\leq \e^{-R-2}$.}

	Next we prove \rmii.  
	Observe that, given a point $x$ in $S$ and a half ball $b_R(y)$ containing~ {$x$, the cylinder $\cC_R(y)$
	contains $b_R(y)$, hence $x$.}  By \rmi\ there exists an admissible cylinder $\cC$ containing $\cC_R(y)$, with
	$\mod{\cC} \leq C \, \mod{\cC_R(y)}$.   Notice that 
	$$
	\begin{aligned}
		\frac{1}{\mod{b_R(y)}} \, \int_{b_R(y)} \mod{f} \wrt \mu
		& \leq \frac{\mod{\cC_R(y)}}{\mod{b_R(y)}} \,\frac{\mod{\cC}}{\mod{\cC_R(y)}} \,\frac{1}{\mod{\cC}} \, 
			\int_{\cC} \mod{f} \wrt \mu \\
		& \leq \frac{C}{\mod{\cC}} \, \int_{\cC} \mod{f} \wrt \mu;
	\end{aligned}
	$$
	we have used Proposition~\ref{p: properties setsNA} in the last inequality.  
	Taking the supremum of both sides with respect to all half balls containing~$x$ we obtain the required estimate. 
\end{proof}

For any cylinder $\cC$ in $\fC_\infty$ we denote by ${ \be} (\cC)$ its base.
For each nonnegative integer $k$, consider the $k$-th ``horizontal'' slice of a cylinder in $\fC_\infty$ 
centred at the origin, given by: 
$$
\beta_k \big(\cC_R(e)\big)
:=\Big\{(X, Z, a) \in \cC_R(e): k<\log \frac{a}{\e^{-R}} \leq k+1\Big\}.
$$
Notice that for any nonnegative integer $k$ and every number $R>1$
\begin{equation}\label{eq:slices NA}
	\e^{k \nu } \, \bigmod{\beta_k(\cC_R(e))} 
	= \bigmod{\beta_0(\cC_R(e))} 
	\quad \text{and} \quad
	\mod{\cC_R(e)} 
	= \frac{1}{1-\e^{-\nu }}\,  \bigmod{\beta_0(\cC_R(e))}.
\end{equation}
Indeed, we have
\begin{align*}
	\mod{\beta_k(\cC_R(e))}
	&=\int_{\cG (X, { Z}) \leq 1} {\wrt X}{\wrt Z} \int_{\e^{-R+k}}^{\e^{-R+k+1}} a^{-\nu -1} \wrt a \\
	&=\e^{-k\nu } \, \int_{\cG (X, { Z}) \leq 1} {\wrt X}{\wrt Z}\, \frac{1}{\nu }(1-\e^{-\nu })\, \e^{\nu R} \\
	&= \e^{-k\nu} \, \bigmod{\be_0 \big(\cC_R(e)\big)},
\end{align*}
which is the first equality in \eqref{eq:slices NA}. Summing it over $k$ gives the second one.

Let us now consider ``horizontal'' slices for cylinders centred at any point $n_0a_0$ (recall the latter 
are described in \eqref{eq:cylinder_n0a0}); the slices are given by
\begin{equation}\label{eq:k-th slice}
\beta_k\big(\cC_R(n_0 a_0)\big)
	:= \Big\{n a \in \cC_R(n_0 a_0):\,  k<\log \frac{a}{a_0 \e^{-R}} \leq  k+1\Big\},
\end{equation}
still satisfying the properties in \eqref{eq:slices NA}.

\section{Proof of Theorem~\ref{t: main}~\rmi-\rmii} \label{s: The upper half-plane}

\subsection{Main estimate}

In the sequel we shall write $\cU$ instead of $\cN^{\fC_\infty'}$ for simplicity.
Suppose that $\al>0$, and consider, for every $f$ in $\lu{S}$, the level set 
$$
E_{\cU\! f}(\al) := \big\{x\in S: \cU\! f (x) > \al \big\}.
$$
We shall often write $E(\al)$ instead of $E_{\cU\! f}(\al)$.

From Proposition~\ref{p: admissibleNA}~\rmii\ we know that there exists a constant $C$ such that 
$$
\cN^{\frb_\infty}\! f 
\leq C\, \cU\! f.  
$$

\begin{theorem} \label{t: main cU}
	The following hold:
	\begin{enumerate}
		\item[\itemno1]
			there exists a positive constant {$C$} such that 
			$$
				\bigmod{\{x\in S: \cU\! f(x)>\al \}}
				\leq {C} \, \int_{S} \, \frac{\mod{f}}{\al} \, \log\Big({2} 
					+ \frac{\mod{f}}{{\al}} \Big) \wrt \mu
			$$
			for every $\al >0$ and every $f$ {for which the right hand side is finite}; 
		\item[\itemno2]
			$\cU$ is bounded on $\lp{S}$ for every $p$ in $(1,\infty]$.
	\end{enumerate}
\end{theorem}

\noindent
Clearly Theorem~\ref{t: main}~\rmi-\rmii\ follows directly from this, the estimate above for $\cN^{\frb_\infty}\! f$,
and the fact that we already know that $\cN^{\fb_0}$ is of weak type $(1,1)$.  


The proof of Theorem~\ref{t: main cU} requires some preliminary  {results}, 
which are stated in Proposition~\ref{p: horintNA} and Proposition~\ref{p: overlapping NA}.

\subsection{Overlapping properties.} 

By the { classical} result of C\'ordoba and Fefferman \cite[Proposition~1]{CF} 
{  in order to get $L^p$ estimates for the uncentred Hardy--Littlewood maximal operator in $\BR^n$ with respect to a family 
of a collection $\fF$ of bounded open sets, it suffices to control the ``overlapping number in the $L^{p'}$ norm'' 
of certain subfamilies of $\fF$.  We adopt a similar strategy on Damek--Ricci spaces.  We shall estimate the ``overlapping number'' 
of certain subfamilies of $\fC_\infty'$ in Proposition~\ref{p: overlapping NA} below. }

Suppose that $u$ is a positive number. 
Clearly the horocycle $\cH_u$ can be identified with $N$ via the map { $nu \mapsto n$}.  
In the proof of the following proposition we shall consider the metric measure spaces $(\cH_u, d_{N}, \mu_u)$, where we 
abuse the notation and write 
$$
d_N\big((X,Z,u), (X',Z', u) \big)
:= d_N\big((X,Z), (X',Z') \big);
$$  
here, $d_N$ is the distance \eqref{eq:distN} on $N$ associated to the gauge $\cG$, and $\mu_u$ is the measure 
$$
	\!\wrt \mu_u(X,Z,u) 
	= \frac{\wrt X \wrt Z}{u^{\nu +1}}.  
$$
Recall the basic covering lemma \cite[Theorem~1.2]{He}, which applies, in particular, to any collection of balls of 
uniformly bounded diameters in $(\cH_u, d_{N}, \mu_u)$.  The following proposition is an important consequence of that lemma.

\begin{proposition}\label{p: horintNA}
	Suppose that $\fF$ is a collection of cylinders in $\fC_\infty'$ such that all their bases have 
	uniformly bounded volumes and are contained in the same horocycle.  Then there exists a 
	subfamily $\fF_d$ of $\fF$ consisting of pairwise disjoint cylinders such that
   	\begin{align*}
		\mu\Big(\bigcup_{\cC\in \fF} \, \cC\Big) 
		\leq  { 5^\nu}\, \mu\Big(\bigcup_{\cC \in \fF_d} \, \cC\Big). 
    	\end{align*} 
\end{proposition}

\begin{proof}
	Clearly two cylinders in $\fF$ are disjoint if and only 
	if their bases are disjoint.  Then it is natural to consider the collection of cylinders 
	$$
		\big\{\be_0(\cC): \cC \in \fF\big\}.
	$$
	By \eqref{eq:k-th slice} and \eqref{eq:vertical_dist} each slice $\be_0(\cC)$ has geodesic height one.  Their 
	bases ${ \be}(\cC)$ { are balls (with respect to the metric $d_N$) belonging} 
	to the same horocycle, $\cH_u$ say.  { Since their volumes are uniformly bounded,}
	 {the basic} covering lemma { \cite[Theorem~1.2]{He}} applies to the collection 
	$\big\{{ \be}(\cC): \cC\in \fF\big\}$.  
	Hence there exists a subfamily $\fF_d$ of $\fF$ such that the sets in $\big\{{ \be}(\cC): \cC \in \fF_d\big\}$ 
	are pairwise disjoint, and 
	$$
	\begin{aligned}
		\mu_u \Big(\bigcup_{{ \cC} \in \fF} \, { \be}({ \cC} )\Big)  
		\leq  \mu_u \Big(\bigcup_{{ \cC}  \in \fF_d} \, { 5\, \be}({ \cC} )\Big)
		=     { 5^\nu} \, \sum_{{ \cC}  \in \fF_d} \, 
			\mu_u \big({ \be}({ \cC} )\big).
	\end{aligned}
	$$
	Clearly the cylinders in $\fF_d$ are pairwise disjoint.  Furthermore
	$$
	\begin{aligned}
		\mu \Big(\bigcup_{\cC\in \fF} \, \cC\Big)  
		& =    \mu_u \Big(\bigcup_{\cC\in \fF} \, { \be}(\cC)\Big) \,\, \int_u^\infty  \frac{\wrt y}{y^{\nu+1}} \\
		& \leq \frac{{ 5^\nu}}{ {\nu}\, u^\nu} \, \sum_{\cC \in \fF_d} \, 
			\mu_u \big({ \be}(\cC)\big)  \\
		& =     { 5^\nu} \, \sum_{\cC \in \fF_d} \, \mu(\cC), 
	\end{aligned}
	$$
	as required.  
	\end{proof}

\begin{definition} \label{def: maximality NA}
	\rm{
	A collection $\fF$ of  {admissible cylinders} { (see Definition~\ref{def: cC admissible})}
	is \textit{maximal} if 
	\begin{enumerate}
		\item[\itemno1]
			it is maximal with respect to inclusion;
		\item[\itemno2]
			any two cylinders in $\fF$, such that their bases lie on the same horocycle, are disjoint.
	\end{enumerate}
}
\end{definition}

For any admissible cylinder $\cC$, we denote by $h(\cC)$ the height of the points belonging to its base ${ \be} (\cC)$.  
Then
$$
h\big( \cC_R(n_0a_0)\big)
= a_0 \, \e^{-R},
$$
so $\log h\big( \cC_R(n_0a_0)\big) = \log a_0-R$ is an integer.   

Given a collection $\fG$ of cylinders in $\fC_\infty'$, denote by $G$ their union. 
Define the \textit{overlapping number}~$\Om$ of the family $\fG$ by 
$$
\Om(x)
:= \sharp \,\{\cC \in \fG: \cC\ni x\} 
\quant x \in G,
$$ 
and put $\Om$ equal to zero outside $G$.  For every positive integer $k$, set 
$$
\Om_k := \big\{x \in G: \Om(x) = k\big\}.
$$

\begin{proposition} \label{p: overlapping NA}
	Suppose that $\fG$ is a \textit{finite} collection of maximal (in the sense of Definition~\ref{def: maximality NA}) 
	cylinders in $\fC_\infty'$, and denote by $G$ their union. 
	Then 
    	\[
    		|\Om_k|
		\leq \frac{\e^{2\nu}}{\e^\nu -1}\, |G| \, \e^{-k}.
    	\]
\end{proposition}

\begin{proof}
	For $x$ in $G$, denote by $\cC_1,\ldots,\cC_{\Om(x)}$ the (distinct) cylinders in~$\fG$ that contain~$x$,
	and by $h_1,\ldots,h_{\Om(x)}$ the heights of their bases, which we list in decreasing order. Since all the cylinders $\cC_1,\ldots,\cC_{\Om(x)}$ contain $x$, hence they are not disjoint,  
	their bases must belong to different horocycles, by the definition of maximality.
	Hence the numbers $\log h_1,\ldots,\log h_{\Om(x)}$ must be distinct integers {(see the
	remark after Definition~\ref{def: maximality NA}).  

	Recall also that the slices $\be_0(\cC_1),\ldots,\be_0(\cC_{\Om(x)})$ have Riemannian thickness $1$.  
	Therefore they are mutually disjoint.
	Hence the Riemannian distance between $x$ and the base of $\cC_{\Omega(x)}$ 
	is at least $\Omega(x)-1$.  Indeed,
	$$
		d\big(x, \be(\cC_{\Omega(x)})\big) 
		\geq \log \frac{h(x)}{h_{\Om(x)}} 
		\geq \log \frac{h_1}{h_{\Om(x)}} 
		=    \sum_{j=2}^{\Om(x)} \, \log \frac{h_{j-1}}{h_{j}}  
		\geq \Om(x)-1.
	$$
	}
	
	Now, if $x$ is in $\Omega_k$, then $x$ belongs to exactly $k$ cylinders in $G$.
	The Riemannian distance between $x$ and $\be_0(\cC_{{ \Om(x)}})$ is at least $k-1$ ($k$ distinct integer heights), 
	so that $x$ belongs to 
	$$
		\bigcup_{m \geq k-1}^{\infty} \,  {\beta}_m\big(\cC_{\Omega(x)}\big) .
	$$
	We let $x$ vary in $\Om_k$, and obtain the containment
	$$
	\Om_k 
	\subseteq \bigcup_{\cC\in \fG} \, \, \bigcup_{m\geq k-1}\,  \be_m(\cC).
	$$
 	Now \eqref{eq:slices NA} yields
	$$
       		\Big|\bigcup_{m\geq k-1} \be_m(\cC)\Big|
		\leq \sum_{m= k-1}^{\infty}\, \e^{-m\nu } \, \bigmod{\be_0(\cC)}
		\leq \frac{\e^{2\nu -k}}{\e^\nu -1} \,\, \bigmod{\be_0(\cC)}.
	$$
	Hence
	$$
	\bigmod{\Om_k}
	\leq \frac{\e^{2\nu -k}}{\e^\nu -1} \,\,\, \sum_{\cC\in \fG} \, \bigmod{\be_0(\cC)}.
	$$
	Since by assumption the cylinders in $\fG$ are maximal, their bases are disjoint.  Therefore
	$$
	\sum_{\cC\in \fG} \, \bigmod{\be_0(\cC)}
	= \Bigmod{\bigcup_{\cC\in \fG} \, \be_0(\cC)}
	\leq \mod{G}.
	$$
	Thus,
   	\begin{equation*}
		\bigmod{\Om_k}
		\leq  \frac{\e^{2\nu -k}}{\e^\nu -1} \,\, \mod{G},
   	\end{equation*}
	as required.
\end{proof}

\begin{remark} 
	\rm{
	Notice that the estimate in Proposition~\ref{p: overlapping NA} has the following interesting consequence.  
	For every $r$ in $[1,\infty)$
	\begin{equation} \label{f: Lorentz}
	\Bignorm{\sum_{\cC\in\fG} \One_\cC}{r}
	\leq A_r \, \norm{\One_G}{r},
	\end{equation}
	where $\ds A_r^r := \frac{\e^{2\nu}}{\e^\nu-1} \,\, \sum_{k=1}^\infty \, k^r \, \e^{-k}$.

	Indeed, 
	$$
    		\int_{G}\, \Om^r \wrt\mu
		=    \sum_{k=1}^\infty \, k^r\, \bigmod{\Om_k}
		\leq \frac{\e^{2\nu}}{\e^\nu-1} \, \sum_{k=1}^\infty \, k^r \, \e^{-k} \,\, \mod{G},  
	$$
	which is equivalent to the required estimate.
	}
\end{remark}

Our main estimate, Theorem~\ref{t: main}~\rmi, requires the following lemma.  

\begin{lemma} \label{l: Young}
	Suppose that $\psi$ is a $C^1$ strictly increasing function on $[0,\infty)$ such that $\psi(0) = 0$.  Then for all positive
	numbers $a,$ $b$ and $\la$ the following inequality holds:
	$$
		ab \leq \la \, \int_0^a \, \psi(u) \wrt u + \la \int_0^{b/\la} \, \psi^{-1}(v) \wrt v.
	$$
\end{lemma}

\begin{proof}
	Young's inequality (see \cite[p. 379]{MPF}) states that if $f$ is an increasing continuous function 
	on $[0,a]$, with $f(0) = 0$, and $b$ is in $\big(0,f(a)\big)$, then 
	$$
		ab \leq \int_0^a \, f(t) \wrt t + \int_0^b\, f^{-1}(v) \wrt v.
	$$
	The required formula follows from this by taking $\la\, \psi$ instead of $f$, and changing variables ($v = u/\la$) 
	in the second integral.  
\end{proof}

Consider the case where $\psi(u):=\e^{u / 2}-1$.  Then $\ds \psi^{-1}(v) = 2 \log (v+1)$, 
and Lemma~\ref{l: Young} yields
$$
\begin{aligned}
ab
	& \leq \la \,\int_0^a  {(\e^{t/ 2}-1)} \wrt t + 2 \la \, \int_0^{b/\la} \, \log (v+1) \wrt v \\ 
	& = \la\, \big(2 \e^{a/2}-a-2\big)  
		+ 2\la\, \Big[\left(\frac{b}{\la}+1\right) \log \left(\frac{b}{\la}+1\right)-\frac{b}{\la}\Big].
\end{aligned}
$$
Since $2\, \e^{a/2} - a -2 \leq 2\, {(\e^{a/2}-1)}$ and 
$$
	(u+1)\, \log(u+1) - u \leq u \, \log(u+1)
	\quant u \in [0,\infty),
$$
we get 
\begin{equation} \label{f: ab less}
	ab \leq 2\,\la \, {(\e^{a/2}-1)} + 2 \,b \,\log\Big(\frac{b}{\la}+1\Big).  
\end{equation}

\begin{proof} (of Theorem~\ref{t: main cU})
%
	First we prove \rmi. 
	Observe that if $x$ belongs to $E(\al)$, then there exists an admissible cylinder $\cC$ containing~$x$ such that  
	\begin{equation} \label{f: average al}
	\frac{1}{\bigmod{\cC}} \, \int_{\cC} \bigmod{f} \wrt \mu
	> \al.  
	\end{equation}
	Then $\cC\subseteq E(\al)$.  This entails that $E(\al)$ can be written as a union of 
	 {admissible} cylinders~$\cC$ for which \eqref{f: average al} holds.  
	Furthermore, if $\cC$ is one of these cylinders and if $f$ belongs to $\lp{S}$ for some $p$ in $(1,\infty)$, 
	then \eqref{f: average al} and H\"older's inequality imply that 
	\begin{equation} \label{f: p estimate}
	\mod{\cC} < \frac{\bignormto{f}{p}{p}}{\al^p}.
	\end{equation}
	Since $\cC$ is a translate of a cylinder $\cC_R(e)$ for some $R>1$, 
	$$
		\mod{\cC} = { \frac{\om}{\nu}}\, \e^{\nu R}>{ \frac{\om}{\nu}}\, \e^{\nu }
	$$
	{ (see Proposition~\ref{p: properties setsNA})}.
	 {This} implies that $\al < { (\nu/\om)}^{1/p}\,\e^{-\nu/p}\,\norm{f}{p}$ and 
	$$
	R < \frac{1}{\nu}\, \log \Big({ \frac{\nu}{\om}}\,\frac{\normto{f}{p}{p}}{\al^p} \Big).
	$$
	Now recall that for $\cC$ in $\fC_\infty'$ the associated number $R$ is an integer greater than one.   
	Therefore, there exist at most $\ga-2$ such numbers, where $\ga$ denotes the integer part of the number on the
	right-hand side of the previous inequality.

	Clearly $E(\al)$ is the union of cylinders $\cC$ such that 
	\eqref{f: average al} holds.  We can consider  {only those admissible cylinders covering 
	$E(\al)$ and satisfying \eqref{f: p estimate}} that are maximal with respect to inclusion.  
	We denote by $\fG(\al)$ the collection of such cylinders.  

	Denote by $\fF^j(\al)$ the subcollection of all $\cC$ in $\fG(\al)$ with $h(\cC) = { \e^j}$.  
	By Proposition~\ref{p: horintNA}, for each integer~$j$ there exists a disjoint subcollection $\fF_d^j(\al)$ of $\fF^j(\al)$ 
	with the property that 
	\begin{equation} \label{f: disjoint fFdj}
		\Bigmod{\bigcup_{\cC\in \fF^j(\al)} \, \cC}
		\leq  { 5^\nu}\, \, \Bigmod{\bigcup_{\cC \in \fF_d^j(\al)} \, \cC}. 
	\end{equation}
	The family 
	$
	\ds\fF(\al)
	:= \bigcup_{j\in\BZ} \, \fF_d^j(\al) 
	$ 
	is maximal in the sense of Definition~\ref{def: maximality NA}.  Now
	\begin{equation} \label{f: intermediate fF}
	\begin{aligned}
		\mod{E(\al)} 
		& \leq \sum_{j\in\BZ} \,\, \Bigmod{\!\bigcup_{\cC\in\fF^j(\al)} \, \cC } \\
		& \leq { 5^\nu}\, \sum_{j\in\BZ} \,\, \Bigmod{\!\bigcup_{\cC\in\fF_d^j(\al)} \, \cC} \\
		& \leq { 5^\nu}\, \sum_{j\in\BZ} \,\, \sum_{\! \cC\in\fF_d^j(\al)} \, \mod{\cC};
	\end{aligned}
	\end{equation}
	the first and the third inequality above follow from the subadditivity of~$\mu$ and the second is a consequence of 
	\eqref{f: disjoint fFdj}.  By \eqref{eq:slices NA} the inequality 
	$\ds \mod{\cC} \leq \frac{\e^\nu}{\e^\nu-1}\, \be_0(\cC)$ holds.  By combining this, the pairwise disjointness 
	of the sets $\be_0(\cC)$ as $\cC$ 
	varies in $\ds\bigcup_{j\in\BZ} \fF_d^j(\al)$, and \eqref{f: intermediate fF}, we obtain that 
	$$
	\begin{aligned}
		\mod{E(\al)} 
		& \leq  \frac{({ 5}\e)^\nu}{\e^\nu-1} \, \sum_{j\in\BZ} \,\, 
			\sum_{\! \cC\in\fF_d^j(\al)} \, \mod{\be_0(\cC)} \\
		& =     \frac{({ 5} \e)^\nu}{\e^\nu-1} \, \Bigmod{\bigcup_{j\in\BZ} \,\, 
			\bigcup_{\! \cC\in\fF_d^j(\al)} \, \be_0(\cC)} \\
		& \leq  \frac{({ 5} \e)^\nu}{\e^\nu-1} \, \Bigmod{\bigcup_{\! \cC\in\fF(\al)} \, \cC}.
	\end{aligned}
	$$

	Suppose that $K$ is a compact subset of $E(\al)$.  Then there is a finite subcollection $\fF_K(\al)$ of cylinders
	in $\fF(\al)$ whose union covers $K$.  

	We argue much as in the proof of Proposition~\ref{p: overlapping NA}, with $\fF_K(\al)$ in place of~$\fG$.  
	We set $\ds G_K(\al) := \bigcup_{\cC\in \fF_K(\al)}\, \cC$ and define the \textit{overlapping number}~$\Om$ of the family 
	$\fF_K(\al)$ by 
	$$
	\Om(x)
	:= \sharp \,\{\cC \in \fF_K(\al): \cC\ni x\} 
	\quant x \in G_K(\al),
	$$ 
	and put $\Om$ equal to zero outside $G_K(\al)$.  For every positive integer $k$ set $\Om_k := \big\{x \in G: \Om(x) = k\big\}$.
	Proposition~\ref{p: overlapping NA} yields the estimate 
   	\begin{equation}\label{claim I}
		\bigmod{\Om_k}
		\leq  \frac{\e^{2\nu}}{\e^{{\nu}}-1} \,\, \mod{G_K(\al)} \, \e^{-k}.
   	\end{equation}
	Much as in the proof of \cite[Proposition 1]{CF} (see also \cite[Proposition~4.3]{MS}), observe that 
	$$
	\begin{aligned}
		\bigmod{G_K(\al)}
		 & \leq  \sum_{\cC\in \fF_K(\al)} \mod{\cC} \\
		 & \leq \frac{1}{\al}\,  \sum_{\cC\in \fF_K(\al)}\,  \int_{\cC} \mod{f} \wrt \mu \\
		 & \leq   \int_{S}  \,\frac{\mod{f}}{\al} \,\, \sum_{\cC\in \fF_K(\al)} \One_\cC\wrt \mu.
	\end{aligned}
	$$
	In order to obtain cleaner formulae, set 
	$$
	w := \sum_{\cC\in \fF_K(\al)} \One_\cC.
	$$    
	The inequality \eqref{f: ab less}, with $\ds a := w$ and $\ds b := \frac{\mod{f}}{\al}$, yields
	$$
		\int_{S}  \,\frac{\mod{f}}{\al} \,\, w \wrt \mu
		\leq 2\, \la \int_{S} \, {(\e^{w /2}-1)} \wrt \mu + 2 \,
			\int_{S} \, \frac{\mod{f}}{\al} \, \log\Big(\frac{\mod{f}}{\al\la}+1\Big) \wrt \mu.
	$$
	Observe that 
	$$
	\begin{aligned}
		\int_{S} \, {(\e^{w /2}-1)} \wrt \mu \, 
		& {\leq}     \sum_{k=1}^\infty \, \e^{k/2} \, \mod{\Om_k} \\
		& \leq \frac{\e^{2\nu}}{\e^\nu-1} \,  \sum_{k=1}^\infty \, \e^{-k/2} \, \mod{G_K(\al)} \\
		& =  \frac{\e^{2\nu}}{(\e^\nu-1) \, (\sqrt \e -1)} \, \mod{G_K(\al)}  {.}
	\end{aligned}
	$$
	By combining the formulae above, we get that
	$$
	\begin{aligned}
		\bigmod{G_K(\al)}
		& \leq 2\, \la \, \frac{\e^{2\nu}}{(\e^\nu-1) \, (\sqrt \e -1)} \, \mod{G_K(\al)}  
			+ 2 \int_{S} \, \frac{\mod{f}}{\al} \, \log\Big(\frac{\mod{f}}{\al\la}+1\Big) \wrt \mu. 
	\end{aligned}
	$$
	Choose 
	$$
	\la 
	= \frac{1}{4} \, \frac{(\e^\nu-1) \, (\sqrt \e -1)}{\e^{2\nu}}.
	$$   
	Then, 
	$$
		\bigmod{G_K(\al)}
		\leq 4 \, \int_{S} \, \frac{\mod{f}}{\al} \, \log\Big(\frac{\mod{f}}{\al\la}+1\Big) \wrt \mu.  
	$$
	Now the inner regularity of $\mu$ and the fact that $\la$ is independent of $K$ yields
	$$
		{\bigmod{E(\al)}
		\leq 4 \, \int_{S} \, \frac{\mod{f}}{\al} \, \log\Big(\frac{\mod{f}}{\al\la}+1\Big) \wrt \mu}.  
	$$
	{Notice that $\la<1$, so that  
	$$
	\log\Big(\frac{\mod{f}}{\al\la}+1\Big)
	\leq \log\la^{-1} + \log \Big(\frac{\mod{f}}{\al}+2\Big),
	$$
	and
	$$
	\begin{aligned}
		\bigmod{E(\al)}
		& \leq 4 \, \log{\la^{-1}}\, \int_S \, \frac{\mod{f}}{\al} \wrt\mu
			+ 4 \, \int_S \, \frac{\mod{f}}{\al} \, \log \Big(\frac{\mod{f}}{\al}+2\Big) \wrt\mu \\			
		& \leq 4 \Big[\frac{\log{\la^{-1}}}{\log 2} + 1\Big]\, 
			\int_S \, \frac{\mod{f}}{\al} \, \log \Big(\frac{\mod{f}}{\al}+2\Big) \wrt\mu,
	\end{aligned}	
	$$
	as required.  
	}

	Next we prove \rmii.  Given a function $f$ in $\lu{S}$ it is convenient to set 
	$$
		F_\al := \big\{x\in S: \mod{f(x)} > \al \big\}
		\quad\hbox{and}\quad
		F^\al := \big\{x\in S: \mod{f(x)} \leq \al \big\}.  
	$$
	Correspondingly, we write 
	$$
	f_\al := \One_{F_\al} \, f
	\quad\hbox{and} \quad
	f^\al := \One_{F^\al} \, f. 
	$$
	{ Recall that $\cU$ is short for $\cN^{\fC_\infty'}$.}
	Clearly $f = f_\al + f^\al$, and $\cU\! f \leq \cU\! f_\al + \cU\! f^\al$. 
	Since $\cU$ is a contraction on $\ly{S}$, $\cU\! f^\al \leq \al$, whence 
	$$
		\{\cU\! f>2\al\}  \subseteq \{\cU\! f_\al > \al \}.
	$$
	Set $E (\al) := \big\{x \in S: \cU\! f > \al \big\}$. 
	Now Theorem~\ref{t: main cU}~\rmi\ and Tonelli's Theorem imply that
	$$
	\begin{aligned}
		\bignormto{\cU\! f}{p}{p}
		& =    p\, \ioty \, \mod{E(\al)} \, \al^{p-1} \, \wrt \al \\
		& \leq C\, p\, \ioty \!\al^{p-1} \,\wrt \al \,  \int_{S} \, \frac{\mod{f_\al(x)}}{\al} \, 
			\log\Big({2}+ \frac{\mod{f_\al(x)}}{{\al}} \Big) \wrt \mu(x) \\
		& \leq C\, p\, \int_{S}  \!\wrt \mu(x) \, \int_0^{\mod{f(x)}}   \,  \, \frac{\mod{f(x)}}{\al} \, 
			\log\Big({2}+ \frac{\mod{f(x)}}{{\al}} \Big) \, \al^{p-1} \,\wrt \al.
	\end{aligned}
	$$
	Changing variables ($\be =  \mod{f(x)}/\al$) transforms the latter double integral into 
	$$
		\int_{S}  \!\wrt \mu(x) \, \mod{f(x)}^p \int_1^\infty   \, \be^{-p} \, \log({2} 
			+ {\be}\big) \wrt \be
		=  {C_{{p}}}\, \bignormto{f}{p}{p}. 
	$$
	The required estimate follows by combining the last two formulae.  	
\end{proof}

\section{Proof of Theorem~\ref{t: main}~\rmiii} \label{s: The upper half plane iii}

By {Proposition~\ref{p: inclusions T and bNA}} the operators $\cN^{\frb_\infty}$ and $\cN^{\fT_\infty}$ 
are pointwise equivalent.  Since $\cN^{\frb_0}$ is of weak type $(1,1)$, in order to prove Theorem~\ref{t: main}~\rmiii\ 
it suffices to prove the following.

\begin{proposition} \label{p: uncentred infinite rectangle primed}
	{Theorem~\ref{t: main}~(i) does not hold if we replace $\Phi$ 
	by another Young function $\Psi$ on $[0,\infty)$ such that 
	$\ds 
	\lim_{t\to+\infty} \, \frac{\Psi(t)}{\Phi(t)} = 0.
	$}
\end{proposition}

{\begin{proof}
	Denote by $B$ a ball in $NA$ centred at $e$ and contained in~$T(w_1)$, where $w_1 := (0,0,\e)$, with volume $\leq 1$.
	Such a ball exists, because $T(w_1)$ is an open set containing $e$.  For notational convenience, we set 
	\begin{equation} \label{f: eta and Eeta}
		\eta
		:= \cN^{\fT_\infty} \One_{B}
		\quad\hbox{and}\quad
		E_\eta(\al)
		:= \big\{\eta > \al \big\}.  
	\end{equation}
	In Lemma~\ref{l: preliminary opt} below we shall prove that there exists a positive constant~$c'$ such that
	\begin{equation} \label{f: main est Eeta}
		\bigmod{E_\eta(\al)}
		\geq c'\, \frac{1}{\al} \,\log\frac{1}{\al}   
	\end{equation}
	for all $\al$ small enough.  

	We proceed by \textit{reductio ad absurdum}.
	If the operator $\cN^{\fb_\infty}$ satisfied a bound like that in Theorem~\ref{t: main}~\rmi, but with 
	$\Psi$ in place of $\Phi$, then there would exist a constant~$C$ such that
	$$
		\bigmod{E_\eta(\al)}
		\leq C \, \int_S \, \Psi\big(\mod{\One_B}/\al\big) \wrt \mu
	$$
	for all $\al$ sufficiently small.  Since $\Psi$ is a Young function, $\Psi(0) = 0$, so that 
	$\Psi\big(\mod{\One_B}/\al\big)$ vanishes off $B$.  Therefore 
	the integral above can be restricted to~$B$ and it is equal to $V_B \, \Psi(1/\al)$, which is at most $\Psi(1/\al)$.  
	By combining this and \eqref{f: main est Eeta}, we obtain 
	$$
		\Phi(1/\al) 
		\asymp \frac{1}{\al} \, \log \frac{1}{\al}
		\leq C \, \Psi\big(1/\al\big)
		\quant \al \, \, \hbox{small},
	$$
	which contradicts the assumption that $\ds \lim_{t\to+\infty} \, \frac{\Psi(t)}{\Phi(t)} = 0$. 
\end{proof}

It remains to prove \eqref{f: main est Eeta}.  Recall that the trigona we consider are open sets.
For each nonnegative integer $j$ consider the point $w_j := (0,0,\e^j)$ in~$S$. 

\begin{lemma} \label{l: gambas}
	Suppose that $a$ is in $(0,1)$.  Then $\big(T(w_1) \setminus T(e)\big)\cap \cH_a$ contains the annulus
	$$
		\big\{na: n \in \cB_{\sqrt{\e-1}} (e_0) \setminus \cB_1(e_0) \big\}. 
	$$
\end{lemma}

\begin{proof}
	By Proposition~\ref{p: shadow on the boundary}, we know that each point $na$ in 
	$\partial (T(e)) \cap \cH_{a}$ satisfies
	$$
		\al_1(a)
		\leq \cG(n) 
		\leq \al_2(a),
	$$
	where $\al_1$ and $\al_2$ are defined in \eqref{f: al1 al2}.  Furthermore, notice that
	$$
	\partial\big(T(w_1)\big) \cap \cH_{a} 
		= w_1 \cdot\big[\partial\big(T(e)\big) \cap \cH_{a/\e}\big].
	$$
	Using the calculation above and the gauge homogeneity, we see that every $na$ in 
	$\partial(T(w_1)) \cap \cH_{a}$ satisfies
	$$
		\sqrt{\e} \,\, \al_1(a/\e)
		\leq \cG(n) 
		\leq \sqrt{\e}\,\, \al_2(a/\e).
	$$
	Notice that $\sqrt{\e}\,\, \al_1(a/\e) > \al_2(a)$.  Thus, $\big(T(w_1) \setminus T(e)\big)\cap \cH_a$ contains the annulus
	\begin{align*}
		\big\{na: n \in \cB_{\sqrt{\e}\,\, \al_1(a/\e)}(e_0) \setminus \cB_{\al_2(a)}(e_0)\big\}. 
	\end{align*}
	Finally, notice that $\al_2(a) < 1$ and $\sqrt{\e}\,\, \al_1(a/\e)>\sqrt{\e-1}$.  This, together with the previous 
	containment yields the required result.
\end{proof}

We retain the notation established in \eqref{f: eta and Eeta}.

\begin{lemma} \label{l: preliminary opt}
	There exists a positive constant $c'$ such that 
			\begin{equation} \label{f: lower bound}
				\bigmod{E_{\eta}(\al)}
				\geq c'\, \frac{1}{\al} \, \log \frac{1}{\al}
			\end{equation}
	for $\al$ small enough.  
\end{lemma}

\begin{proof}
Clearly $E_\eta(\al)$ is the union of maximal trigona $T$ satisfying 
$$
	\frac{1}{\mod{T}} \, \int_{T} \, \One_B \wrt \mu
	> \al.
$$
Then 
$$
	\mod{T}
	<    \frac{1}{\al} \, \mod{T\cap B}
	\leq \frac{V_B}{\al}. 
$$
Notice that any trigonon $T'$ containing $B$ and satisfying $\mod{T'} < V_B/\al$ is contained in 
$E_\eta(\al)$.  

Recall that, by \eqref{f: trigonon more}, 
\begin{equation} \label{f: est vol TRz}
	\bigmod{T_R(z)}
	\leq C_\nu\, \e^{\nu R}
\quant z \in S \quant R>1.
\end{equation}
For every $\al$ positive and small enough, set
\begin{equation} \label{f: Ralpha}
R_\al
	:= \frac{1}{\nu} \, \Big\lfloor \log\frac{V_B}{2\e^{4\nu r_B}C_\nu\, \al}\Big\rfloor,
\end{equation}
where $r_B$ denotes the radius of $B$.  
Then \eqref{f: est vol TRz} yields
$$
	\bigmod{T_{R_\al+4r_B}(z)} 
	\leq C_\nu\, \e^{\nu (R_\al+4r_B)}
	\leq C_\nu \, \frac{\e^{4\nu r_B}\, V_B}{2\e^{4\nu r_B} C_\nu\,\al}
	< \frac{V_B}{\al}
	\quant z \in S.
$$
Thus, if we show that for each $j$ in $\{1, \ldots, R_\al-1\}$ the trigonon $T_{R_\al+4r_B}(w_j)$ contains $B$, 
we can conclude that $T_{R_\al+4r_B}(w_j)$ is contained in $E_\eta(\al)$.

For $\al$ small enough we have that $R_{\al}>2$.  Recall that~$B$ is contained in~$T(w_1)$, 
by construction.  Since $T(w_j)$ contains $T(w_1)$ for $j\geq 2$, $B$ is also contained in $T(w_j)$.  

Furthermore, the point in $B$ with lowest height is $\e^{-r_B} \, e$.  Therefore~$B$ is contained in
$\Pi_{\e^{-r_B}}$.  It follows that the trigona $T_{R_\al+4r_B}(w_j)$, $j\in \{1,\ldots,R_\al\}$, contain $B$.

Thus, we have proved that all these trigona are contained in $E_{\eta}(\al)$.  Hence 
$$
E_{\eta}(\al)
\supset T_{R_\al+4r_B}(w_1) \cup \, 
\bigcup_{j=2}^{R_\al-1}\big(T_{R_\al+4r_B}(w_j)\setminus T_{R_\al+4r_B}(w_{j-1}) \big).
$$
We \textit{claim} that  there exists a positive constant $\kappa$, independent of $j$, such that 
       \begin{equation}\label{eq:estFS1 S}
	\bigmod{T_{R_\al+4r_B}(w_j)\setminus T_{R_\al+4r_B}(w_{j-1})}
       \geq \kappa\, \e^{\nu R_\al}.
       \end{equation}

	Given the claim, we see  {from \eqref{eq:estFS1 S} and the inclusion above} that 
	$$
        \bigmod{E_{\eta}(\al)}
		\geq \kappa\, \e^{\nu R_{\al}} \,(R_\al-2)
	$$
	for all $\al$ small enough.  Clearly, there exists a positive number $c'$ such that
	$$
		\kappa\, \e^{\nu R_{\al}} \,(R_\al-2) 
		\geq \frac{\kappa}{2}\, \e^{\nu R_{\al}} \,{R_\al}
                \geq \frac{c'}{\al} \, \log \frac{1}{\al}
	$$
	as required to prove \eqref{f: main est Eeta}.
	
	Finally, we prove the claim \eqref{eq:estFS1 S}.  Notice that 
	$$
	T_{R_\al+4r_B}(w_j)\setminus T_{R_\al+4r_B}(w_{j-1})
	= w_{j-1} \cdot \Big(T_{R_\al+4r_B}(w_1)\setminus T_{R_\al+4r_B}(e)\Big),
	$$
	Since $A$ acts isometrically on $S$, we have
	$$
	\bigmod{T_{R_\al+4r_B}(w_j)\setminus T_{R_\al+4r_B}(w_{j-1})}
	= \bigmod{T_{R_\al+4r_B}(w_1)\setminus T_{R_\al+4r_B}(e)}
	$$
	for every $j$ in $\{1,\ldots,R_\al-1\}$.  Thus, it suffices to prove \eqref{eq:estFS1 S} for $j=1$.

	To prove this, consider the annulus 
	$$
	\cA
	:=  \cB_{\sqrt{\e-1}} (e_0) \setminus \cB_1(e_0).
	$$  
	Observe that 
	$$
	\big(T_R(w_1)\backslash T_R(e)\big)
	= \big(T(w_1)\backslash T(e)\big)\cap \Pi_{\e^{1-R}}.  
	$$ 
	By Lemma~\ref{l: gambas}, this set contains $\cA \times \big(\e^{1-R}, \e^{2-R}\big)$ for $R>3$.  Therefore
	\[
		\begin{aligned}
		|T_R(w_1)\backslash T_R(e)| 
			& \geq \int_{\cA} \wrt n  \, \, 
				\int_{\e^{1-R}}^{\e^{2-R}}\frac{\wrt s}{s^{\nu+1}}  \\
			& = \om\, \big((\e-1)^{\nu/2} - 1\big) \,\, \frac{\e^{-\nu}}{\nu} \,(1-\e^{-\nu}) \,  \e^{\nu R},
		\end{aligned}
	\]
	as required.

	Taking $R$ equal to $R_\al+4r_B$ $(\geq R_\al)$, proves claim \eqref{eq:estFS1 S}.  
\end{proof}
}

\section{A variant of $\cN^{\fb_\infty}$} \label{s: Variants}
In this section we consider a variant of $\cN^{\fb_\infty}$, inspired by \cite[Proposition~5.3]{MS}, 
where a variant of the uncentred triangular maximal operator on trees is considered. 

In order to avoid long computations we present our result only on $\BH$.
For each $z$ in $\mathbb{H}$ and $R\geq 1$, consider the \emph{modified} half ball
\[
b'_R(z)
:= b_R(z)\cup B_1\big(\Re z+i \e^R \Im z\big).
\]
Notice that the point $\Re z+i \e^R \Im z$ lies on the { upward} vertical geodesic {  starting
at} $z$, and 
$$
d\big(z,\Re z+i \e^R \Im z)
= R.
$$
We denote by $\frb_\infty'$ the collection of all modified balls $b_R'(z)$, where $z$ is in $\BH$ and $R$ 
is at least $1$, and denote by $\cN^{\frb_\infty'}$ the associated uncentred maximal operator.

{ 
It is straighforward to check that $\cN^{\fb_\infty'}$ is dominated by $C \, \cN^\fB$, which, by a celebrated result
of Ionescu \cite{I1}, is bounded on $\lp{\BH}$ for all $p>2$.
We show that $\cN^{\fb_\infty'}$ is unbounded on $\lp{\BH}$ if $p$ is in $[1,2]$.  
This illustrates how sensitive maximal operators are to the shape of the sets with respect to which we take averages.
The reason for which $\cN^{\fb_\infty'}$ is unbounded on $\lp{\BH}$ for all $p$ in $[1,2]$ is that 
families of maximal subsets in $\frb_\infty'$ have ``worse overlapping properties'' than the corresponding sets in $\frb_\infty$. }

{ It is straightforward to check that if $R\geq 1$, then $b_R(z)$ and $B_1\big(\Re z+i \e^R \Im z\big)$ are disjoint}, 
and so
\[
	\mod{b'_R(z)}
= \mod{ b_R(z)}+\bigmod{B_1(\Re z+i \e^R \Im z)}
\asymp \e^{R} 
\quant z\in \BH \quant R\geq 1.
\]

{ 
\begin{theorem} \label{t: modified half balls}
	The operator $\cN^{b_\infty'}$ is unbounded on $\lp{\BH}$ for every $p$ in $[1,2]$.
\end{theorem}

\begin{proof}
Consider the hyperbolic balls $B_1(-1+i)$ and $B_1(1+i)$.  They are the Euclidean discs with centres $-1+i\cosh 1$
and $1+i\cosh 1$, and radius $\sinh 1$.  Define
\begin{equation} \label{f: inters satellites}
	\cR
	:= B_1(-1+i) \cap B_1(1+i).   
\end{equation}
An elementary calculation shows that $\cR$ is nonempty.   

On each horocycle $\cH_{\e^{-2^{\ell}}}$, with $\ell \geq 0$, consider a maximal collection of nonintersecting horizontal translates 
of $b_{2^\ell}\bigl(i \e^{-2^\ell}\bigr)$, such that the first and the last of them are 
\[
b_{2^{\ell}}\bigl(-1+ i \e^{-2^\ell}\bigr)
\quad\hbox{and}\quad
b_{2^{\ell}}\bigl(1 + i \e^{-2^\ell}\bigr).
\]
We denote by $n_\ell$ their number. 
Each of these half balls are tangent to the horocycle $\cH_{\e^{-2^{\ell+1}}}$.  Therefore they have empty interesection
with the half balls with centres at height $\e^{-2^{\ell+1}}$.

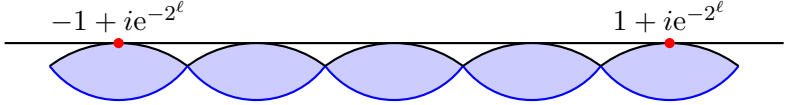
\begin{figure}[ht]
\centering
\begin{tikzpicture}[scale=1.5]

\def\n{5}
\def\R{0.7}

\pgfmathsetmacro{\sinhR}{sinh(\R)}
\pgfmathsetmacro{\coshR}{cosh(\R)}
\pgfmathsetmacro{\tanhR}{tanh(\R)}

\pgfmathsetmacro{\totalwidth}{2*(\n-1)*\tanhR}

\draw[thick] ({-\totalwidth/2 - 1},1) -- ({\totalwidth/2 + 1},1);


\foreach \k in {0,...,\numexpr\n-1} {

    \pgfmathsetmacro{\x}{2*\k*\tanhR - \totalwidth/2}

    \begin{scope}
        \clip (\x,0) circle (1);
        \fill[blue!20] (\x,\coshR) circle (\sinhR);
    \end{scope}

    \begin{scope}
        \clip (\x,0) circle (1);
        \draw[blue, thick] (\x,\coshR) circle (\sinhR);
    \end{scope}

    \begin{scope}
        \clip (\x,\coshR) circle (\sinhR);
        \draw[black, thick] (\x,0) circle (1);
    \end{scope}
}

\pgfmathsetmacro{\leftx}{-\totalwidth/2}
\filldraw[red] (\leftx,1) circle (0.04);
	\node[above] at (\leftx,1) {$-1+i\e^{-2^\ell}$};

\pgfmathsetmacro{\rightx}{\totalwidth/2}
\filldraw[red] (\rightx,1) circle (0.04);
\node[above] at (\rightx,1) {$1+i\e^{-2^\ell}$};

\end{tikzpicture}
\vspace{-60pt}
\caption{A packing of half balls at height $\e^{-2^\ell}$.}
\label{tiling}
\end{figure}

Moreover, \eqref{f: qzRm} and \eqref{f: qzRp} imply that
each half ball with centre at height~$\e^{-2^\ell}$ and radius $2^\ell$ contains a Euclidean segment of Euclidean length 
\[
\rho_{\ell} := 2 \, \e^{-2^\ell} \, \tanh 2^{\ell}.
\]
Clearly
$$
2\, (\tanh1)\, \e^{-2^{\ell}} 
	\leq \rho_{\ell} 
	\leq 2\, \e^{-2^\ell}
	\quant \ell \geq 0.
$$
This and a straightforward calculation yield
\[
n_{\ell} = \left\lfloor \frac{1}{\rho_\ell} \right\rfloor + 1 \asymp \exp(2^{\ell}).
\]

We next consider the $n_{\ell}$ associated modified half balls, 
\begin{equation}\label{eq:half balls ell}
b'_{2^{\ell}}\bigl(-1 + i \e^{-2^\ell}\bigr), \ldots, b'_{2^{\ell}}\bigl(1 + i \e^{-2^\ell}\bigr).
\end{equation}
We observe that the first and last ``satellite'' balls are 
\[
B_1\bigl(-1 + i\bigr)
\quad \text{and} \quad
B_1\bigl(1 + i\bigr),
\]
and that 
$$
\bigcup_{x\in [-1,1]} \, B_1\bigl(x + i\bigr) 
= \cR,
$$
where $\cR$ is defined in \eqref{f: inters satellites}.


For any $\ell\geq0$, we define $E_{\ell}$ as the union of the $n_{\ell}$ half balls associated to the 
modified half balls given in \eqref{eq:half balls ell}. 
Each of these half balls has radius $2^{\ell}$, hence area comparable to $ \exp(2^{\ell})$.  
We conclude that
\begin{equation}\label{eq:vol E ell}
	\bigmod{E_{\ell}} \asymp n_{\ell} \, \exp(2^{\ell})\asymp  \exp(2^{\ell+1}).
\end{equation}

We observe that there exists a positive constant $c$ such that 
$$
\cN^{\frb_\infty'} \mathbbm{1}_\cR (z)
\geq c\, \exp(-2^\ell)
\quant z\in E_\ell \quant \ell \geq 0.
$$
Indeed, let $z$ belong to $E_\ell$.  Then $z$ belongs to some half ball $b$ in $E_\ell$.  The corresponding modified half ball $b'$
contains $\cR$.  Therefore
\[
\cN^{\frb_\infty'} \mathbbm{1}_\cR(z)
\ge \frac{1}{\mod{b'}} \, \int_{b'} \, \mathbbm{1}_\cR  \wrt\mu
= \frac{\mod{\cR}}{\mod{b'}}
\geq c\, \exp(-2^{\ell}).
\]
As a consequence, 
\begin{align*}
\|
	\cN^{\frb_\infty'} \mathbbm{1}_\cR\|_p^p
	& \ge \sum_{\ell=0}^{\infty}\,\, \mod{E_\ell}\, \exp(-p \, 2^{\ell}) \\
& \geq c\,  \sum_{\ell=0}^{\infty} \, \exp\bigl((2-p) \, 2^{\ell}\bigr),
\end{align*}
where for the last inequality we used \eqref{eq:vol E ell}. 
If $1\leq p\leq 2$, then the last sum is infinite, as required.
\end{proof}
}

\end{document}